\documentclass[11pt]{amsart}
\usepackage[margin=1in]{geometry}

\usepackage{amssymb}
\usepackage{amsthm}
\usepackage{amsmath}
\usepackage{mathrsfs}
\usepackage{amsbsy}
\usepackage[all]{xy}
\usepackage{bm}
\usepackage{hyperref}
\usepackage{tikz}
\usepackage{array}
\usepackage{float}
\usepackage{enumerate}
\usepackage{xcolor}
\usepackage{hhline}
\allowdisplaybreaks
\usepackage{cite}

\newcommand{\sgn}{\operatorname{sgn}}
\DeclareMathOperator{\FS}{\mathsf{FS}}
\newcommand{\Id}{\operatorname{Id}}
\newcommand{\rev}{\operatorname{rev}}
\newcommand{\cupdot}{\mathbin{\mathaccent\cdot\cup}}

\newenvironment{enumerate*}%
  {\begin{enumerate}[(I)]%
    \setlength{\itemsep}{10pt}%
    \setlength{\parskip}{0pt}}%
  {\end{enumerate}}

\newtheorem{theorem}{Theorem}[section]
\newtheorem{proposition}[theorem]{Proposition}

\newtheorem{conjecture}[theorem]{Conjecture}
\newtheorem{question}[theorem]{Question}
\newtheorem{problem}[theorem]{Problem}
\newtheorem{lemma}[theorem]{Lemma}

\theoremstyle{definition}
\newtheorem{definition}[theorem]{Definition}

\newtheorem{example}[theorem]{Example}

\begin{document}

\title{Typical and Extremal Aspects of Friends-and-Strangers Graphs}
\subjclass[2010]{}

\author[Noga Alon, Colin Defant, and Noah Kravitz]{Noga Alon}
\address{Department of Mathematics, Princeton University, Princeton, NJ 08544, USA and Schools of Mathematics and Computer Science, Tel Aviv University, Tel Aviv 69978, Israel}
\email{nalon@math.princeton.edu}
\author[]{Colin Defant}
\address[]{Department of Mathematics, Princeton University, Princeton, NJ 08544}
\email{cdefant@princeton.edu}
\author[]{Noah Kravitz}
\address[]{Department of Mathematics, Princeton University, Princeton, NJ 08544}
\email{nkravitz@princeton.edu}


\begin{abstract}
Given graphs $X$ and $Y$ with vertex sets $V(X)$ and $V(Y)$ of the same cardinality, the friends-and-strangers graph $\FS(X,Y)$ is the graph whose vertex set consists of all bijections $\sigma:V(X)\to V(Y)$, where two bijections $\sigma$ and $\sigma'$ are adjacent if they agree everywhere except for two adjacent vertices $a,b \in V(X)$ such that $\sigma(a)$ and $\sigma(b)$ are adjacent in $Y$.  The most fundamental question that one can ask about these friends-and-strangers graphs is whether or not they are connected; we address this problem from two different perspectives.  First, we address the case of ``typical'' $X$ and $Y$ by proving that if $X$ and $Y$ are independent Erd\H{o}s-R\'enyi random graphs with $n$ vertices and edge probability $p$, then the threshold probability guaranteeing the connectedness of $\FS(X,Y)$ with high probability is $p=n^{-1/2+o(1)}$.  Second, we address the case of ``extremal'' $X$ and $Y$ by proving that the smallest minimum degree of the $n$-vertex graphs $X$ and $Y$ that guarantees the connectedness of $\FS(X,Y)$ is between $3n/5+O(1)$ and $9n/14+O(1)$. When $X$ and $Y$ are bipartite, a parity obstruction forces $\FS(X,Y)$ to be disconnected. In this bipartite setting, we prove analogous ``typical'' and ``extremal'' results concerning when $\FS(X,Y)$ has exactly $2$ connected components; for the extremal question, we obtain a nearly exact result.
\end{abstract}

\maketitle

\section{Introduction}
\subsection{Background}

The second and third authors \cite{friends} recently introduced a general problem concerning friends and strangers walking on graphs.  Given simple graphs $X$ and $Y$ on $n$ vertices, we define the \emph{friends-and-strangers graph} of $X$ and $Y$, denoted $\FS(X,Y)$, as follows. The vertex set of $\FS(X,Y)$ is the set of all bijections $\sigma: V(X) \to V(Y)$ from the vertex set of $X$ to the vertex set of $Y$; two bijections $\sigma$ and $\sigma'$ are connected by an edge if and only if $X$ contains an edge $\{a,b\}$ such that $\{\sigma(a),\sigma(b)\}$ is an edge in $Y$, $\sigma(a)=\sigma'(b)$, $\sigma(b)=\sigma'(a)$, and
$\sigma(c)=\sigma'(c)$ for all $c\in V(X)\setminus\{a,b\}$.
In other words, we connect $\sigma$ and $\sigma'$ if they differ only at a pair of adjacent vertices such that the images of these vertices under $\sigma$ are adjacent in $Y$.  In this case, the operation that transforms $\sigma$ into $\sigma'$ is called an \emph{$(X,Y)$-friendly swap}. We will sometime refer to this operation as an \emph{$(X,Y)$-friendly swap across $\{a,b\}$} when we wish to specify the edge of $X$ over which the swap takes place. 

The friends-and-strangers graph $\FS(X,Y)$ has the following non-technical interpretation.  Identify $n$ different people with the vertices of $Y$.  Say that two such people are friends with each other if they are adjacent in $Y$, and say that they are strangers otherwise.  Now, suppose that these people are standing on the vertices of $X$ so that each vertex has exactly one person standing on it.  At each point in time, two friends standing on adjacent vertices of $X$ may swap places, but two strangers may not.  It is natural to ask how various configurations can be reached from other configurations when we allow multiple such swaps, and this is precisely the information that is encoded in $\FS(X,Y)$.  In particular, the connected components of $\FS(X,Y)$ correspond to the equivalence classes of mutually-reachable configurations.

This framework is quite general, and several special cases have received attention in the past in other contexts.  For instance, Stanley \cite{stanley} studied the connected components of $\FS(\mathsf{Path}_n, \mathsf{Path}_n)$; the graph $\FS(K_n, Y)$ is the Cayley graph of $\mathfrak S_n$ generated by the transpositions corresponding to edges of $Y$; the famous $15$-puzzle can be interpreted in terms of $\FS(\mathsf{Star}_{16},\mathsf{Grid}_{4\times 4})$; and Wilson \cite{wilson}, generalizing the $15$-puzzle, studied the connected components of $\FS(\mathsf{Star}_n,Y)$ for arbitrary graphs $Y$.  In \cite{friends}, the second and third authors established several general properties of $\FS(X,Y)$ and investigated necessary and sufficient conditions for $\FS(X,Y)$ to be connected. For arbitrary graphs $Y$, they also characterized the connected components of $\FS(\mathsf{Path}_n, Y)$ in terms of acyclic orientations of the complement of $Y$ and characterized the connected components of $\FS(\mathsf{Cycle}_n, Y)$ in terms of toric acyclic orientations (also called toric partial orders) of the complement of $Y$. In this paper, we will continue the theme of determining when $\FS(X,Y)$ is connected. While the article \cite{friends} focused on exact results of an enumerative/algebraic flavor, we will focus here on more probabilistic and extremal questions. 

\subsection{Main results}
Most previous work on $\FS(X,Y)$ has focused on the case where $X$ (or $Y$) is a particular highly-structured graph.  One natural and new question concerns the connected components of $\FS(X,Y)$ when $X$ and $Y$ are random graphs. We denote by $\mathcal G(n,p)$ the probability space of Erd\H{o}s-R\'{e}nyi random edge-subgraphs of the complete graph $K_n$ in which each edge appears with probability $p$.  If we choose $X$ and $Y$ independently from $\mathcal G(n,p)$, what values of $p$ guarantee that, with high probability, the graph $\FS(X,Y)$ is either connected or disconnected?  We answer this question by finding, up to a multiplicative factor of $n^{o(1)}$, the threshold for $p$ at which $\FS(X,Y)$ changes from disconnected with high probability to connected with high probability.  (As usual, we say that an event occurs \emph{with high probability} if its probability of occurring tends to $1$ as the size of the graph involved tends to $\infty$.) 

\begin{theorem}\label{thm:random}
Fix some small $\varepsilon>0$. Let $X$ and $Y$ be independently-chosen random graphs in $\mathcal G(n,p)$, where $p=p(n)$ depends on $n$. If \[p\leq\frac{2^{-1/2}-\varepsilon}{n^{1/2}},\] then $\FS(X,Y)$ is disconnected with high probability. If \[p\geq\frac{\exp(2(\log n)^{2/3})}{n^{1/2}},\] then $\FS(X,Y)$ is connected with high probability.
\end{theorem}

The first inequality in this theorem comes from the threshold for $\FS(X,Y)$ having isolated vertices; it seems that (as in the usual case of a binomial random graph) this local obstruction to connectedness tells essentially the whole story.

An \emph{Erd\H{o}s-R\'enyi random edge-subgraph of $K_{r,r}$ with edge probability $p$} is an edge-subgraph of the complete bipartite graph $K_{r,r}$ in which each edge appears with probability $p$ and the events that different edges appear are independent. Let $\mathcal G(K_{r,r},p)$ be the probability space of these random graphs. In Proposition~\ref{prop:bipartite}, we will see that there is a simple parity obstruction that keeps $\FS(X,Y)$ from being connected if both $X$ and $Y$ are $n$-vertex bipartite graphs for $n\geq 3$. One might wonder when $\FS(X,Y)$ has exactly $2$ components if $X$ and $Y$ are independently-chosen graphs in $\mathcal G(K_{r,r},p)$. 

\begin{theorem}\label{thm:randombipartite}
Fix some small $\varepsilon>0$. Let $X$ and $Y$ be independently-chosen random graphs in $\mathcal G(K_{r,r},p)$, where $p=p(r)$ depends on $r$. If \[p\leq\frac{1-\varepsilon}{r^{1/2}},\] then $\FS(X,Y)$ has more than $2$ connected components with high probability. If \[p\geq\frac{5 (\log r)^{1/10}}{r^{3/10}},\] then $\FS(X,Y)$ has exactly $2$ connected components with high probability.
\end{theorem}
  
Next, from a more extremal point of view, we examine what minimum-degree condition on $X$ and $Y$ suffices to guarantee the connectedness of $\FS(X,Y)$.  Let $\delta(G)$ denote the minimum degree of the graph $G$.

\begin{theorem}\label{Thm:mindegree}
For each $n\geq 1$, let $d_n$ denote the smallest nonnegative integer such that whenever $X$ and $Y$ are $n$-vertex graphs with $\delta(X)\geq d_n$ and $\delta(Y)\geq d_n$, the friends-and-strangers graph $\FS(X,Y)$ is connected. We have \[d_n\geq \frac{3}{5}n-2.\] If $n\geq 16$, then \[d_n \leq \frac{9}{14}n+2.\] 
\end{theorem}

We remark that the threshold at which isolated vertices disappear is lower than the threshold at which $\FS(X,Y)$ becomes connected. Indeed, it is not difficult to see that the graph $\FS(X,Y)$ cannot have isolated vertices if $\delta(X)\geq n/2$ and $\delta(Y)\geq n/2$. 

Of course, we can ask the same question in the case where $X$ and $Y$ are bipartite. In this case, we are able to obtain upper and lower bounds that are extremely close to each other.

\begin{theorem}\label{Thm:bipartitemindegree}
For each $r\geq 2$, let $d_{r,r}$ be the smallest nonnegative integer such that whenever $X$ and $Y$ are edge-subgraphs of $K_{r,r}$ with $\delta(X)\geq d_{r,r}$ and $\delta(Y)\geq d_{r,r}$, the friends-and-strangers graph $\FS(X,Y)$ has exactly $2$ connected components. We have \[\left\lceil\frac{3r+1}{4}\right\rceil\leq d_{r,r}\leq\left\lceil\frac{3r+2}{4}\right\rceil.\]
\end{theorem}

The lower and upper bounds for $d_{r,r}$ in Theorem~\ref{Thm:bipartitemindegree} differ by $1$ when $r\equiv 1\pmod 4$ and are equal when $r\not\equiv 1\pmod 4$. The lower bound is actually the cutoff for $\FS(X,Y)$ to avoid isolated vertices; unlike in the non-bipartite case, this cutoff is essentially the same as the cutoff for $\FS(X,Y)$ to become connected (even though we do not know how to prove that the cutoffs are exactly the same when $r\equiv 1\pmod 4$).

\subsection{Structure of the paper}

In Section~\ref{Sec:Preliminaries}, we fix notation and terminology and establish several important facts concerning friends-and-strangers graphs that will be used throughout the rest of the article. Sections~\ref{Sec:Random}, \ref{Sec:RandomBipartite}, \ref{Sec:MinDegree}, and \ref{Sec:MinDegreeBipartite} are devoted to proving Theorems~\ref{thm:random}, \ref{thm:randombipartite}, \ref{Thm:mindegree}, and \ref{Thm:bipartitemindegree}, respectively. In Section~\ref{Sec:Conclusion}, we raise several open questions and conjectures. 

\section{Preliminaries}\label{Sec:Preliminaries}

We begin by recording several results, definitions, and preliminary observations. For additional information about friends-and-strangers graphs, we refer the reader to \cite{friends}. 

\subsection{Basic notation and terminology}
\begin{itemize}
\item We write $[n]$ for the set $\{1,\ldots,n\}$.
\item All graphs in this paper are assumed to be simple. We let $V(G)$ and $E(G)$ denote the vertex set and edge set, respectively, of a graph $G$. 
\item We let $\delta(G)$ denote the minimum degree of a graph $G$.
\item If $G$ is a graph and $S\subseteq V(G)$, then $G \vert_S$ denotes the induced subgraph of $G$ on $S$.
\item We let $K_n$ denote the complete graph on $n$ vertices and let $K_{r,s}$ denote the complete bipartite graph with partite sets of sizes $r$ and $s$. Let $\mathsf{Star}_n$ denote the star graph $K_{1,n-1}$. 
\item If $G$ is a graph and $v$ is one of its vertices, then $N(v)=\{u \in V(G): \{v,u\} \in E(G)\}$ denotes the \emph{open neighborhood} of $v$ and $N[v]=N(v) \cup \{v\}$ denotes the \emph{closed neighborhood} of $v$.  If $S\subseteq V(G)$, then we similarly define the open and closed neighborhoods $N(S)=\bigcup_{v \in S}N(v)$ and $N[S]=\bigcup_{v \in S}N[v]$, respectively.
\item If $\Sigma$ is a finite sequence, then $\rev(\Sigma)$ denotes the reverse of $\Sigma$. 
\end{itemize}

Let $X$ and $Y$ be $n$-vertex graphs, and let $\{u,v\}$ be an edge in $Y$. Suppose $\sigma,\sigma':V(X)\to V(Y)$ are bijections such that $\sigma'$ is obtained from $\sigma$ by performing an $(X,Y)$-friendly swap across the edge $\{\sigma^{-1}(u),\sigma^{-1}(v)\}\in E(X)$. (Recall that this means that $\sigma$ and $\sigma'$ agree on $V(X)\setminus\{\sigma^{-1}(u),\sigma^{-1}(v)\}$ and disagree on the set $\{\sigma^{-1}(u),\sigma^{-1}(v)\}$.) It will be convenient to refer to the $(X,Y)$-friendly swap transforming $\sigma$ into $\sigma'$ by writing simply $uv$. If $\sigma_0, \sigma_1, \ldots, \sigma_r$ is a sequence of bijections, where each $\sigma_i$ is obtained from $\sigma_{i-1}$ by the $(X,Y)$-friendly swap $u_iv_i$, then we say $u_1v_1, u_2v_2, \ldots, u_rv_r$ is a sequence of $(X,Y)$-friendly swaps that transforms $\sigma$ into $\sigma'$. We say this sequence of swaps \emph{involves} the vertices $u_1,v_1,u_2,v_2,\ldots,u_r,v_r$ (and does not involve vertices that are not in this list). 

\begin{example}\label{Exam1}
Let \[X=\begin{array}{l}\includegraphics[height=.77cm]{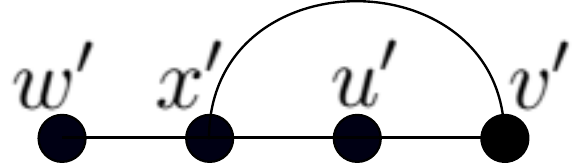}\end{array}\quad\text{and}\quad Y=\begin{array}{l}\includegraphics[height=1.45cm]{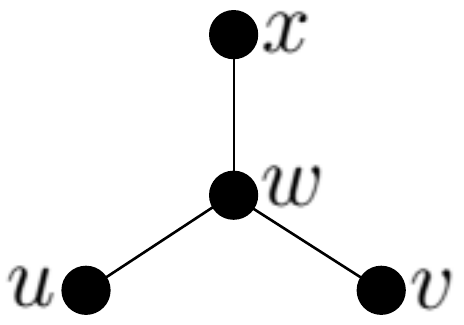}\end{array}.\] Let us denote a bijection $\tau:V(X)\to V(Y)$ by a diagram of $X$ with each vertex $a$ also labeled by $\tau(a)$ in red. For instance, the bijection $\sigma$ defined by $\sigma(u')=u$, $\sigma(v')=v$, $\sigma(w')=w$, and $\sigma(x')=x$ is represented by \[\begin{array}{l}\includegraphics[height=1.08cm]{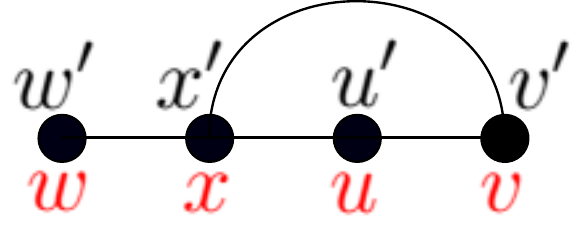}\end{array}.\] The sequence $\Sigma=wx,wu,wv,wu,wx$ of $(X,Y)$-friendly swaps transforms $\sigma$ into the bijection $\sigma'$ defined by $\sigma'(u')=v$, $\sigma'(v')=u$, $\sigma'(w')=w$, and $\sigma'(x')=x$. The application of this sequence to $\sigma$ can be represented pictorially as 
\begin{align*}
&\begin{array}{l}\includegraphics[height=1.08cm]{RandomFriendsPIC8}\end{array}
\xrightarrow{wx}
\begin{array}{l}\includegraphics[height=1.08cm]{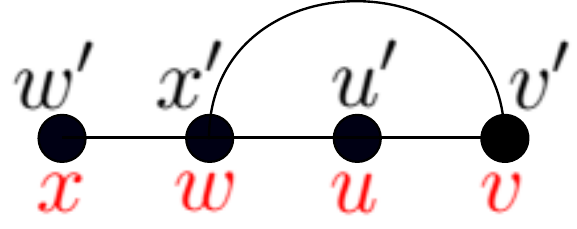}\end{array}
\xrightarrow{wu}
\begin{array}{l}\includegraphics[height=1.08cm]{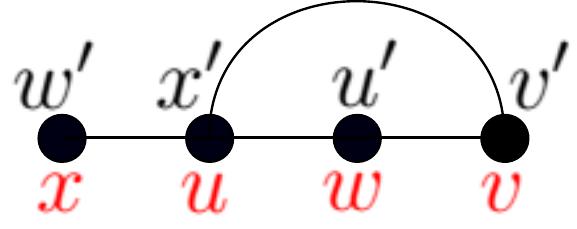}\end{array}\\
\xrightarrow{wv}
&\begin{array}{l}\includegraphics[height=1.08cm]{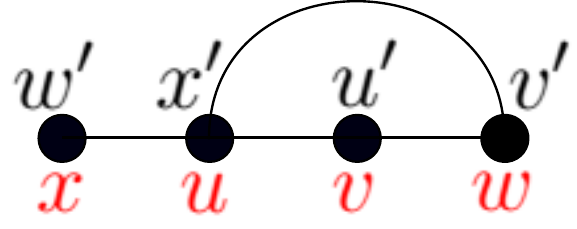}\end{array}
\xrightarrow{wu}
\begin{array}{l}\includegraphics[height=1.08cm]{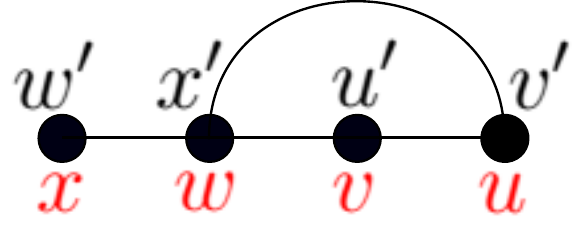}\end{array}
\xrightarrow{wx}
\begin{array}{l}\includegraphics[height=1.08cm]{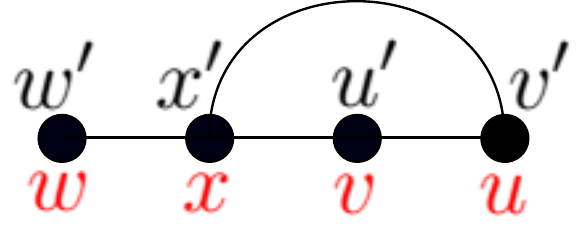}\end{array}.
\end{align*}
 \end{example}

It will be useful to keep in mind that the order of $X$ and $Y$ is somewhat irrelevant because the map $V(\FS(X,Y))\to V(\FS(Y,X))$ given by $\sigma\mapsto\sigma^{-1}$ is a graph isomorphism.

\subsection{Stars and Wilsonian graphs}
A graph $G$ is called \emph{biconnected} if for every $v\in V(G)$, the induced subgraph $G\vert_{V(G)\setminus\{v\}}$ obtained by deleting $v$ from $G$ is connected. We will make frequent use of one of Wilson's results \cite{wilson} about graphs of the form $\FS(\mathsf{Star}_n,Y)$. The statement of this result makes reference to the special graph \[\theta_0=\begin{array}{l}\includegraphics[height=1.38cm]{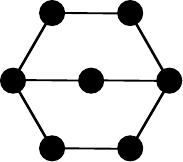}\end{array}.\]

\begin{theorem}[\!\!\cite{wilson}]\label{thm:wilson}
If $Y$ is a graph on $n$ vertices that is biconnected, not a cycle graph with at least $4$ vertices, not isomorphic to $\theta_0$, and not bipartite, then $\FS(\mathsf{Star}_n,Y)$ is connected. 
\end{theorem}

Theorem~\ref{thm:wilson} motivates the following definition. 

\begin{definition}\label{def:wilsonian}
We say a graph $G$ is \emph{Wilsonian} if it is biconnected, non-bipartite, and neither a cycle graph with at least $4$ vertices nor isomorphic to the exceptional graph $\theta_0$. 
\end{definition}

The proof of Theorem~\ref{thm:wilson} is quite involved and is algebraic in nature; we will often use the result as a black box in conjunction with the following sufficient condition for a graph to be Wilsonian (which is clear from the definitions).

\begin{lemma}\label{lem:wilsonian}
If $G$ is an $n$-vertex graph with minimum degree $\delta(G) > n/2$, then $G$ is Wilsonian.
\end{lemma}

\subsection{The case of bipartite $X$ and $Y$}
Given a finite set $A$, we can consider the symmetric group $\mathfrak S_A$ of all bijections $\sigma:A\to A$. We let $\sgn:\mathfrak S_A\to\mathbb Z/2\mathbb Z$ be the sign homomorphism, which is the unique group homomorphism satisfying $\sgn((a\,\, b))=1$ for every transposition $(a\,\,b)\in\mathfrak S_A$. Note that if $A$ and $B$ are two finite sets and $\sigma,\tau:A\to B$ are bijections, then $\sgn(\sigma^{-1}\circ\tau)$ has the same parity as the number of transpositions by which we must multiply $\sigma$ to obtain $\tau$. The following proposition (essentially Proposition 2.7 from \cite{friends}) says that $X$ and $Y$ being bipartite is an obstruction to $\FS(X,Y)$ being connected. 

\begin{proposition}\label{prop:bipartite}
Let $X$ and $Y$ be bipartite graphs on $n$ vertices with vertex bipartitions $\{A_X,B_X\}$ and $\{A_Y,B_Y\}$, respectively.  If the bijections $\sigma$ and $\tau$ are in the same connected component of $\FS(X,Y)$, then
$\sgn(\sigma^{-1}\circ\tau)$ has the same parity as $|\tau(A_X)\cap A_Y|-|\sigma(A_X)\cap A_Y|$. 
\end{proposition}

\begin{proof}
The bijection $\tau$ is obtained from $\sigma$ by performing a sequence of $(X,Y)$-friendly swaps across some edges $\{a_1,b_1\},\ldots,\{a_\ell,b_\ell\}$ (in this order) for some $a_1,\ldots,a_\ell\in A_X$ and $b_1,\ldots,b_\ell\in B_X$. Let $\sigma_0=\sigma$ and $\sigma_i=\sigma_{i-1}\circ(a_i\,\,b_i)$ for all $i\in[\ell]$, so that $\tau=\sigma_\ell$. If $\sigma_{i-1}(a_i)\in A_Y$, then $\sigma_i(a_i)=\sigma_{i-1}(b_i)\in B_Y$, so $\sigma_i(A_X)\cap A_Y=(\sigma_{i-1}(A_X)\cap A_Y)\setminus\{a_i\}$. If $\sigma_{i-1}(a_i)\in B_Y$, then $\sigma_i(a_i)=\sigma_{i-1}(b_i)\in A_Y$, so $\sigma_i(A_X)\cap A_Y=(\sigma_{i-1}(A_X)\cap A_Y)\cup\{a_i\}$. In each of these cases, $|\sigma_i(A_X)\cap A_Y|-|\sigma_{i-1}(A_X)\cap A_Y|$ is either $1$ or $-1$.  Hence, \[|\tau(A_X)\cap A_Y|-|\sigma(A_X)\cap A_Y|=\sum_{i=1}^\ell(|\sigma_i(A_X)\cap A_Y|-|\sigma_{i-1}(A_X)\cap A_Y|)\] has the same parity as $\ell$, which has the same parity as $\sgn(\sigma^{-1}\circ\tau)=\sgn((a_1\,\,b_1)\circ\cdots\circ(a_\ell\,\,b_\ell))$.
\end{proof}
If $X$ and $Y$ are $n$-vertex bipartite graphs with vertex bipartitions $\{A_X,B_X\}$ and $\{A_Y,B_Y\}$, respectively, then we say two bijections $\sigma,\tau:V(X)\to V(Y)$ are \emph{concordant} if $\sgn(\sigma^{-1}\circ\tau)$ has the same parity as $|\tau(A_X)\cap A_Y|-|\sigma(A_X)\cap A_Y|$. It is straightforward to check that concordance is an equivalence relation. An immediate consequence of Proposition~\ref{prop:bipartite} is that if $n\geq 3$, then $\FS(X,Y)$ has at least $2$ connected components. It is natural to ask when this parity obstruction is in fact the only obstacle preventing the graph $\FS(X,Y)$ from being connected. In other words, for various choices of bipartite graphs $X$ and $Y$, we are interested in when $\FS(X,Y)$ has exactly $2$ connected components. As a first step, we show that this is the case when $X$ and $Y$ are complete bipartite graphs that are not both stars. 

\begin{proposition}\label{prop:complete-bipartite}
Let $n \geq 4$, and let $r$ and $s$ be integers satisfying $1 \leq r \leq n-1$ and $2 \leq s \leq n-2$.  Then the graph $\FS(K_{r,n-r},K_{s,n-s})$ has exactly $2$ connected components.
\end{proposition}

\begin{proof}
Let $\{A_X,B_X\}$ and $\{A_Y,B_Y\}$ be the vertex bipartitions of $X$ and $Y$, respectively. It suffices to prove that any two concordant vertices are in the same connected component of $\FS(X,Y)$. We will show this by induction on $n$. The case $n=4$ can be verified by hand (or computer), so assume $n\geq 5$.

If $r=1$ or $r=n-1$, then it follows from Theorem 1 of \cite{wilson} that $\FS(K_{1,n-1},K_{s,n-s})=\FS(\mathsf{Star}_{n},K_{s,n-s})$ has $2$ connected components. In this case, it follows from Proposition~\ref{prop:bipartite} that any two concordant vertices are in the same connected component. Hence, we may assume $2\leq r\leq n-2$. Let $\sigma,\tau:V(X)\to V(Y)$ be concordant bijections. Let us assume for the moment that $\sigma(A_X)\cap\tau(B_X)\neq\emptyset$. This implies that there exists a vertex $y_0\in V(Y)$ such that the vertex $x_0=\sigma^{-1}(y_0)$ is in $A_X$ and the vertex $x_1=\tau^{-1}(y_0)$ is in $B_X$ (hence, $\{x_0,x_1\}\in E(X)$). Let $y_1$ be a vertex that is adjacent to $y_0$ in $Y$. Let $\widetilde\sigma:V(X)\to V(Y)$ be a bijection that is concordant with $\widetilde\sigma$ and satisfies $\widetilde\sigma(x_0)=y_0$ and $\widetilde\sigma(x_1)=y_1$ (such a bijection exists because $n\geq 5$). Let $\widetilde\tau=\widetilde\sigma\circ(x_0\,\,x_1)$ be the bijection obtained by applying an $(X,Y)$-friendly swap across $\{x_0,x_1\}$ to $\widetilde\sigma$. Because $\widetilde\sigma$ is concordant with $\widetilde\tau$ (by Proposition~\ref{prop:bipartite}) and $\sigma$ is concordant with $\widetilde\sigma$ and $\tau$, it follows that $\widetilde\tau$ is concordant with $\tau$. Furthermore, $\widetilde\tau(x_0)=y_1$ and $\widetilde\tau(x_1)=y_0$. 

We claim that $\widetilde\sigma$ is in the same connected component of $\FS(X,Y)$ as $\sigma$ and that $\widetilde\tau$ is in the same connected component as $\tau$. Let $X'=X\vert_{V(X)\setminus\{x_0\}}$ and $Y'=Y\vert_{V(Y)\setminus\{y_0\}}$. Because $\sigma$ and $\widetilde\sigma$ are concordant, the bijections $\sigma\vert_{V(X')}:V(X')\to V(Y')$ and $\widetilde\sigma\vert_{V(X')}:V(X')\to V(Y')$ are concordant as well. It follows by induction that these latter two bijections are in the same connected component of $\FS(X',Y')$, so there is a sequence $\Sigma$ of $(X',Y')$-friendly swaps transforming $\sigma\vert_{V(X')}$ into $\widetilde\sigma\vert_{V(X')}$. We can view $\Sigma$ as a sequence of $(X,Y)$-friendly swaps transforming $\sigma$ into $\widetilde\sigma$, so $\widetilde\sigma$ is in the same connected component as $\sigma$. A completely analogous argument proves that $\widetilde\tau$ is in the same connected component as $\tau$. We defined $\widetilde\tau$ so that it is adjacent to $\widetilde\sigma$ in $\FS(X,Y)$, so $\sigma$ and $\tau$ are in the same connected component of $\FS(X,Y)$.  

Now assume $\sigma(A_X)\cap\tau(B_X)=\emptyset$. We have $\tau(B_X)=V(Y)\setminus\sigma(A_X)=\sigma(V(X)\setminus A_X)=\sigma(B_X)$. There exist $a\in A_X$, $b\in B_X$, and a bijection $\sigma'$ such that $\sigma'$ is obtained by applying an $(X,Y)$-friendly swap across $\{a,b\}$ to $\sigma$. Now $\sigma'(a)=\sigma(b)\in\sigma(B_X)=\tau(B_X)$, so $\sigma'(A_X)\cap\tau(B_X)\neq\emptyset$. By the previous discussion, $\sigma'$ and $\tau$ are in the same connected component of $\FS(X,Y)$. Consequently, $\sigma$ and $\tau$ are in the same connected component. 
\end{proof}

\subsection{Exchangeable pairs}

We now turn to an extension of the notion of an $(X,Y)$-friendly swap.  Let $X$ and $Y$ be $n$-vertex graphs, and fix a bijection $\sigma \in V(\FS(X,Y))$.  Let $u$ and $v$ be distinct vertices of $Y$, and write $u'=\sigma^{-1}(u)$ and $v'=\sigma^{-1}(v)$.  Let $\sigma\circ(u'\,\,v')$ be the bijection that sends $u'$ to $v$, sends $v'$ to $u$, and sends $x$ to $x$ for all $x \in V(X) \setminus \{u',v'\}$.  We say that $u$ and $v$ are \emph{$(X,Y)$-exchangeable from $\sigma$} if $\sigma$ and $\sigma\circ(u'\,\,v')$ are in the same connected component of $\FS(X,Y)$.  In other words, $u$ and $v$ are exchangeable from $\sigma$ if there is a sequence of $(X,Y)$-friendly swaps that, when applied to $\sigma$, has the overall effect of swapping $u$ and $v$ (even if this swap is not itself $(X,Y)$-friendly). If $\Sigma$ is a sequence of $(X,Y)$-friendly swaps that transforms $\sigma$ into $\sigma\circ(u'\,\,v')$, then we say that applying $\Sigma$ to $\sigma$ \emph{exchanges} $u$ and $v$. 
For instance, this is the case in Example~\ref{Exam1}.

As a warm-up, let us establish some basic facts about exchangeability. 

\begin{lemma}\label{lem:exchangeabilityinverse}
Let $X$ and $Y$ be graphs, and let $\sigma\in V(\FS(X,Y))$. Two vertices $u,v\in V(Y)$ are $(X,Y)$-exchangeable from $\sigma$ if and only if $\sigma^{-1}(u)$ and $\sigma^{-1}(v)$ are $(Y,X)$-exchangeable from $\sigma^{-1}$. 
\end{lemma}
\begin{proof}
This statement follows immediately from the observation that the map $\sigma \mapsto \sigma^{-1}$ is a graph isomorphism from $\FS(X,Y)$ to $\FS(Y,X)$.
\end{proof}

The following proposition shows that when we are concerned only with the vertex sets of the connected components of $\FS(X,Y)$, widespread exchangeability can be worth as much as actually having more edges in $Y$.

\begin{proposition}\label{lem:exchangeable}
Let $X$, $Y$, and $\widetilde Y$ be $n$-vertex graphs such that $Y$ is an edge-subgraph of $\widetilde Y$.  Suppose that for every edge $\{u, v\}$ of $\widetilde Y$ and every bijection $\sigma$ satisfying $\{\sigma^{-1}(u), \sigma^{-1}(v)\} \in E(X)$, the vertices $u$ and $v$ are $(X,Y)$-exchangeable from $\sigma$.  Then the connected components of $\FS(X,Y)$ and the connected components of $\FS(X,\widetilde Y)$ have the same vertex sets.  In particular, the number of connected components of $\FS(X, \widetilde Y)$ is equal to the number of connected components of $\FS(X, Y)$.
\end{proposition}

\begin{proof}
It suffices to show that if $\{\sigma, \sigma'\}$ is an edge in $\FS(X, \widetilde Y)$, then $\sigma$ and $\sigma'$ are in the same connected component of $\FS(X,Y)$.  Let $\{\sigma,\sigma'\}\in E(\FS(X,\widetilde Y))$. The bijection $\sigma'$ is obtained from $\sigma$ by an $(X,\widetilde Y)$-friendly swap across some edge $\{u', v'\}$ in $X$; let $u=\sigma(u')$ and $v=\sigma(v')$ so that $\{u,v\} \in E(\widetilde Y)$.  Our hypothesis tells us that $u$ and $v$ are $(X,Y)$-exchangeable from $\sigma$;  that is to say, $\sigma$ and $\sigma'$ are in the same connected component of $\FS(X,Y)$.
\end{proof}

We highlight the special case where $X$ is connected and $\widetilde Y=K_n$.

\begin{lemma}\label{lem:exchangeable-K_n}
Let $X$ and $Y$ be $n$-vertex graphs, and suppose that $X$ is connected.  Suppose that for all distinct vertices $u,v\in V(Y)$ and every bijection $\sigma$ satisfying $\{\sigma^{-1}(u), \sigma^{-1}(v)\} \in E(X)$, the vertices $u$ and $v$ are $(X,Y)$-exchangeable from $\sigma$.  Then $\FS(X,Y)$ is connected.
\end{lemma}
\begin{proof}
It suffices to consider the case in which $V(X)=V(Y)=[n]$ so that $V(\FS(X,Y))$ is the symmetric group $\mathfrak S_n$. Each edge $\{i,j\}$ of $X$ corresponds naturally to the transposition $(i\,\,j)\in\mathfrak S_n$. By Proposition~\ref{lem:exchangeable}, it suffices to show that $\FS(X,K_n)$ is connected. Note that $\FS(X,K_n)$ is isomorphic to the Cayley graph of $\mathfrak S_n$ generated by the transpositions corresponding to the edges of $X$. It is well known \cite[Lemma 3.10.1]{Godsil} and easy to show that this Cayley graph is connected (i.e., the transpositions corresponding to edges of $X$ generate $\mathfrak S_n$) since $X$ is connected. 
\end{proof}

\section{Random Graphs}\label{Sec:Random}

\subsection{Disconnectedness with high probability}

Notice that the following proposition, which is phrased in terms of isolated vertices instead of disconnectedness, is actually stronger than the first statement in Theorem~\ref{thm:random}. 

\begin{proposition}\label{Prop:NoIsolated}
Fix any small $\varepsilon>0$, and let $X$ and $Y$ be independently-chosen random graphs in $\mathcal G(n,p)$. If \[p=p(n)\leq\frac{2^{-1/2}-\varepsilon}{n^{1/2}},\] then the friends-and-strangers graph $\FS(X,Y)$ has an isolated vertex with high probability. 
\end{proposition}

\begin{proof}
There is an extensive literature dealing with edge-disjoint placements of graphs. In particular, Sauer and Spencer \cite{Sauer} (see also Catlin's article \cite{Catlin}) proved that if $G$ and $H$ are $n$-vertex graphs with maximum degrees $\Delta(G)$ and $\Delta(H)$ satisfying $2\Delta(G)\Delta(H)<n$, then there exists a bijection $\sigma:V(G)\to V(H)$ such that for every edge $\{a,b\}$ of $G$, the pair $\{\sigma(a),\sigma(b)\}$ is not an edge in $H$. This is equivalent to the statement that $\sigma$ is an isolated vertex in $\FS(G,H)$. Let us now return to our random graphs $X$ and $Y$. It suffices to consider the case in which $p=(2^{-1/2}-\varepsilon)/n^{1/2}$; in this case, it is well known that $\Delta(X)=pn(1+o(1))$ and $\Delta(Y)=pn(1+o(1))$ with high probability. Consequently, $2\Delta(X)\Delta(Y)=2p^2n^2(1+o(1))\leq 2\left(2^{-1/2}-\varepsilon\right)^2(1+o(1))n<n$ with high probability. 
\end{proof}

\subsection{Connectedness with high probability}\label{Sec:random-upper}

In this section, we prove the second part of Theorem~\ref{thm:random}. In order to do this, we first prove a somewhat technical lemma that allows us to find specific pairs of graphs embedded in pairs of random graphs. This lemma will also be one of our main tools when we analyze random bipartite graphs in Section~\ref{Sec:RandomBipartite}. Let us first introduce some notation and definitions. 

Let $m$ be a positive integer, and let $G$ and $H$ be two graphs on the vertex set $[m]$. Let $X$ and $Y$ be $n$-vertex graphs, and let $\sigma:V(X)\to V(Y)$ be a bijection. Let $V_1,\ldots,V_m$ be a list of $m$ pairwise disjoint sets of vertices of $Y$. We say that the pair of graphs $(G,H)$ is \emph{embeddable in $(X,Y)$ with respect to the sets $V_1,\ldots,V_m$ and the bijection $\sigma$} if there exist vertices $v_i\in V_i$ for all $i\in[m]$ such that for all $i,j\in [m]$, we have 
\[\{i,j\}\in E(H)\implies\{v_i,v_j\}\in E(Y)\quad\text{and}\quad\{i,j\}\in E(G)\implies\{\sigma^{-1}(v_i),\sigma^{-1}(v_j)\}\in E(X).\]

Suppose $q_1,\ldots,q_m$ are nonnegative integers satisfying $q_1+\cdots+q_m\leq n$. We say the pair $(G,H)$ is \emph{$(q_1,\ldots,q_m)$-embeddable in $(X,Y)$} if for every list $V_1,\ldots,V_m$ of pairwise disjoint subsets of $V(Y)$ satisfying $|V_i|=q_i$ for all $i\in[m]$ and every bijection $\sigma:V(X)\to V(Y)$, the pair $(G,H)$ is embeddable in $(X,Y)$ with respect to the sets $V_1,\ldots,V_m$ and the bijection $\sigma$. 

\begin{lemma}\label{Lem:Embeddable}
Let $m,n,q_1,\ldots,q_m$ be positive integers such that $Q:=q_1+\cdots+q_m\leq n$, and let $G$ and $H$ be two graphs on the vertex set $[m]$. For every set $J\subseteq [m]$, let $\beta(J)=|E(G\vert_J)|+|E(H\vert_J)|$. Choose $0\leq p\leq 1$, and let $X$ and $Y$ be independently-chosen random graphs in $\mathcal G(n,p)$. If for every set $J\subseteq [m]$ satisfying $\beta(J)\geq 1$ we have \[p^{\beta(J)}\prod_{j\in J}q_j\geq 3\cdot 2^{m+1}Q\log n,\] then the probability that the pair $(G,H)$ is $(q_1,\ldots,q_m)$-embeddable in $(X,Y)$ is at least $1-n^{-Q}$. 
\end{lemma}

\begin{proof}
We may assume $\beta([m])\geq 1$ since the lemma is trivial otherwise. Fix a list $V_1,\ldots,V_m$ of pairwise disjoint subsets of $V(Y)$ satisfying $|V_i|=q_i$ for all $i\in[m]$ and an injection $\iota:\bigcup_{i\in[m]}V_i\to V(X)$. There are at most $n^{2Q}$ ways to make these choices. Now extend $\iota^{-1}$ arbitrarily to a bijection $\sigma:V(X)\to V(Y)$. Note that whether or not $(G,H)$ is embeddable in $(X,Y)$ with respect to the sets $V_1,\ldots,V_m$ and the bijection $\sigma$ does not depend on the way in which we extended $\iota^{-1}$ to $\sigma$. We will show that the probability that $(G,H)$ is not embeddable in $(X,Y)$ with respect to the sets $V_1,\ldots,V_m$ and the bijection $\sigma$ is at most $n^{-3Q}$; this will imply the desired result. In order to do this, we make use of the Janson Inequalities (as stated in Theorems~8.1.1 and 8.1.2 of \cite{AlonSpencer}). 

Given a tuple $t=(v_1,\ldots,v_m)\in V_1\times\cdots \times V_m$, let $B_t$ be the event that for all $i,j\in [m]$, we have 
\[\{i,j\}\in E(H)\implies\{v_i,v_j\}\in E(Y)\quad\text{and}\quad\{i,j\}\in E(G)\implies\{\sigma^{-1}(v_i),\sigma^{-1}(v_j)\}\in E(X).\] For tuples $t=(v_1,\ldots,v_m)$ and $t'=(v_1',\ldots,v_m')$ in $V_1\times\cdots \times V_m$ and $J\subseteq[m]$, we write $t\sim_J t'$ if $J=\{j\in [m]:v_j=v_j'\}$. We write $t\sim t'$ if there exist $i,j\in[m]$ such that $v_i=v_i'$, $v_j=v_j'$, and $\{i,j\}\in E(G)\cup E(H)$. Observe that $t\sim t'$ if and only if $t\sim_J t'$ for some set $J\subseteq[m]$ satisfying $\beta(J)\geq 1$. Moreover, if $t\not\sim t'$, then the events $B_t$ and $B_{t'}$ are independent.  We define \[\Delta=\sum_{t\sim t'}\text{Pr}[B_t\wedge B_{t'}],\] where $\text{Pr}[B_t\wedge B_{t'}]$ is the probability that $B_t$ and $B_{t'}$ both occur and the sum is over all ordered pairs $(t,t')$ such that $t,t'\in V_1\times\cdots \times V_m$ and $t\sim t'$. Let \[\mu=\sum_{t\in V_1\times\cdots\times V_m}\text{Pr}[B_t]\] be the expected number of the events $B_t$ that occur. We have 
\begin{equation}\label{Eq6}
\mu=p^{\beta([m])}\prod_{j\in [m]}q_j\geq 3\cdot 2^{m+1}Q\log n,
\end{equation} where the inequality follows from our hypothesis with $J=[m]$. 

Observe that we can write 
\begin{equation}\label{Eq7}
\Delta=\sum_{\substack{J\subseteq[m]\\ \beta(J)\geq 1}}\Delta_J,\end{equation} where $\displaystyle\Delta_J=\sum_{t\sim_J t'}\text{Pr}[B_t\wedge B_{t'}]$. For each $J\subseteq[m]$ with $\beta(J)\geq 1$, we have 
\begin{equation}\label{Eq8}
\Delta_J\leq\left(\prod_{j\in J}q_j\right)p^{\beta(J)}\left(\prod_{i\in[m]\setminus J}q_i^2\right)p^{2(\beta([m])-\beta(J))}=\frac{\mu^2}{p^{\beta(J)}\prod_{j\in J}q_j}.
\end{equation} Indeed, the number of ways to choose vertices $v_j=v_j'\in V_j$ for all $j\in J$ is $\prod_{j\in J}q_j$, and the probability that we have 
\[\{i,j\}\in E(H)\implies\{v_i,v_j\}\in E(Y)\quad\text{and}\quad\{i,j\}\in E(G)\implies\{\sigma^{-1}(v_i),\sigma^{-1}(v_j)\}\in E(X)\] for all $i,j\in J$ is $p^{\beta(J)}$. Moreover, the number of ways to choose the (distinct) vertices $v_i,v_i'\in V_i$ for all $i\in[m]\setminus J$ is at most $\prod_{i\in[m]\setminus J}q_i^2$, and the probability that we have 
\[\{i,j\}\in E(H)\implies\{v_i,v_j\},\{v_i',v_j'\}\in E(Y)\] and \[\{i,j\}\in E(G)\implies\{\sigma^{-1}(v_i),\sigma^{-1}(v_j)\},\{\sigma^{-1}(v_i'),\sigma^{-1}(v_j')\}\in E(X)\] for all $(i,j)\in([m]\times[m])\setminus(J\times J)$ is $p^{2(\beta([m])-\beta(J))}$. Combining \eqref{Eq7} and \eqref{Eq8} yields the inequality \begin{equation}\label{Eq9}
\Delta\leq\sum_{\substack{J\subseteq[m]\\ \beta(J)\geq 1}}\frac{\mu^2}{p^{\beta(J)}\prod_{j\in J}q_j}.
\end{equation}

The event that $(G,H)$ is not embeddable in $(X,Y)$ with respect to the sets $V_1,\ldots,V_m$ and the bijection $\sigma$ is the same as the event $\bigwedge_{t\in V_1\times\cdots\times V_m}\overline{B_t}$ that none of the events $B_t$ occur. If $\Delta\leq \mu$, then by the Janson Inequality \cite[Theorem~8.1.1]{AlonSpencer} and \eqref{Eq6}, we have \[\Pr\left[\bigwedge_{t\in 
V_1\times\cdots\times V_m}\overline{B_t}\right]\leq e^{-\mu+\Delta/2}\leq e^{-\mu/2}\leq e^{-(6Q\log n)/2}=n^{-3Q},\] as desired. Now suppose $\Delta\geq \mu$. By \eqref{Eq7}, there must be a subset $J^*\subseteq[m]$ such that $\beta(J^*)\geq 1$ and $\Delta_{J^*}\geq\Delta/2^m$. Using \eqref{Eq8}, we find that \[\frac{\mu^2}{2\Delta}\geq \frac{\mu^2}{2^{m+1}\Delta_{J^*}}\geq\frac{\mu^2}{2^{m+1}}\cdot \frac{p^{\beta(J^*)}\prod_{j\in J^*}q_j}{\mu^2}=\frac{1}{2^{m+1}}p^{\beta(J^*)}\prod_{j\in J^*}q_j.\] Combining this with the hypothesis of the lemma shows that $\mu^2/(2\Delta)\geq 3Q\log n$. We can now use the extended Janson Inequality \cite[Theorem~8.1.2]{AlonSpencer} to find that \[\Pr\left[\bigwedge_{t\in 
V_1\times\cdots\times V_m}\overline{B_t}\right]\leq e^{-\mu^2/(2\Delta)}\leq e^{-3Q\log n}=n^{-3Q},\] as desired.
\end{proof}

Our basic strategy for proving the second part of Theorem~\ref{thm:random} is as follows. Let $n$ be a large integer, and let 
\[
m=\left\lfloor\left(\log n\right)^{2/3}\right\rfloor
.\] (We omit the floor symbols in what follows since doing so will not affect asymptotics.)  We will construct specific graphs $G^*$ and $H^*$ on the vertex set $[m+2]$ and then use repeated applications of Wilson's theorem to prove that the vertices $m+1$ and $m+2$ are $(G^*,H^*)$-exchangeable from the identity bijection $\Id$. We will then consider independently-chosen random graphs $X$ and $Y$ in $\mathcal G(n,p)$ and use Lemma~\ref{Lem:Embeddable} to show that the following holds with high probability: For any fixed vertices $u,v\in V(Y)$ and any bijection $\sigma:V(X)\to V(Y)$ such that $\{\sigma^{-1}(u),\sigma^{-1}(v)\}\in E(X)$, there is a graph embedding $\varphi$ of $G^*$ into $X$ and a graph embedding $\psi$ of $H^*$ into $Y$ such that $\psi(m+1)=u$, $\psi(m+2)=v$, and $\psi\circ\Id=\sigma\circ\varphi$. This will imply that (with high probability) for any such $u$, $v$, and $\sigma$, the vertices $u$ and $v$ are $(X,Y)$-exchangeable from $\sigma$; the proof will then follow from Lemma~\ref{lem:exchangeable-K_n}. 

To begin this endeavor, we describe the graphs $G^*$ and $H^*$, each with vertex set $[m+2]$. Let \[\ell=\left\lfloor \sqrt{m}/2\right\rfloor\] (once again, we omit the floor symbols in what follows). Denote the elements of $[m]$ (written in an arbitrary order) by \[w,x_1,\ldots,x_\ell,y_1,\ldots,y_\ell,z_1,\ldots,z_{m-2\ell-1}.\] The edge set of $H^*$ consists of all edges of the form $\{m+1,x_i\}$, $\{m+2,y_j\}$, $\{w,x_i\}$, $\{w,y_j\}$, or $\{w,z_k\}$ for $1\leq i,j\leq \ell$ and $1\leq k\leq m-2\ell-1$. In other words, if we let $H^{**}$ denote the star graph on the vertex set $[m]$ with center $w$, then $H^*$ is obtained from $H^{**}$ by adding the vertices $m+1$ and $m+2$, along with the additional edges of the form $\{m+1,x_i\}$ and $\{m+2,y_j\}$. To construct $G^*$, we first construct a graph $G^{**}$ on the vertex set $[m]$ by arranging the vertices in $[m]$ along a cycle graph in such a way that the vertices $z_1,\ldots,z_{12}$ appear in this (cyclic) order when we traverse the cycle clockwise. We then add in the edges $\{z_1,z_6\}$, $\{z_2,z_4\}$, $\{z_7,z_{12}\}$, $\{z_8,z_{10}\}$. We also make sure to place the vertices on the cycle in such a way that the following conditions are satisfied: 
\begin{itemize}
    \item The cycle contains the edges $\{z_4,z_5\}$, $\{z_5,z_6\}$, $\{z_{10},z_{11}\}$,  $\{z_{11},z_{12}\}$ (i.e., $z_4,z_5,z_6$ appear consecutively along the cycle, as do $z_{10},z_{11},z_{12}$).  
    \item The clockwise distance along the cycle between $z_3$ and $z_5$ is $\ell-1$, as is the clockwise distance along the cycle between $z_9$ and $z_{11}$. 
    \item The clockwise distance along the cycle from $z_2$ to $z_4$ is an even number. 
    \item The $2\ell+1$ vertices $w,x_1,\ldots,x_\ell,y_1,\ldots,y_\ell$ are placed on the cycle so that the distance in $G^{**}$ between any two of them, as well as the distance in $G^{**}$ between any one of them and any one of the vertices $z_3,z_5,z_9,z_{11}$, is at least $m/(3\ell)$.
    \item The girth of the entire graph $G^{**}$ is at least $m/6$.
\end{itemize}

The graph $G^*$ is obtained from $G^{**}$ by adding the vertices $m+1$ and $m+2$ and the additional edges $\{m+1,m+2\}$, $\{m+1,z_3\}$, $\{m+1,z_{11}\}$, $\{m+2,z_5\}$, $\{m+2,z_9\}$ to $G^{**}$. Figure~\ref{Fig3} shows schematic drawings of the graphs $G^*$ and $H^*$. In what follows, we refer to the cycle in $G^*$ containing all the vertices in $[m]$ as the \emph{large cycle in $G^*$}.  

\begin{figure}[ht]
\begin{center}
\includegraphics[height=7.9cm]{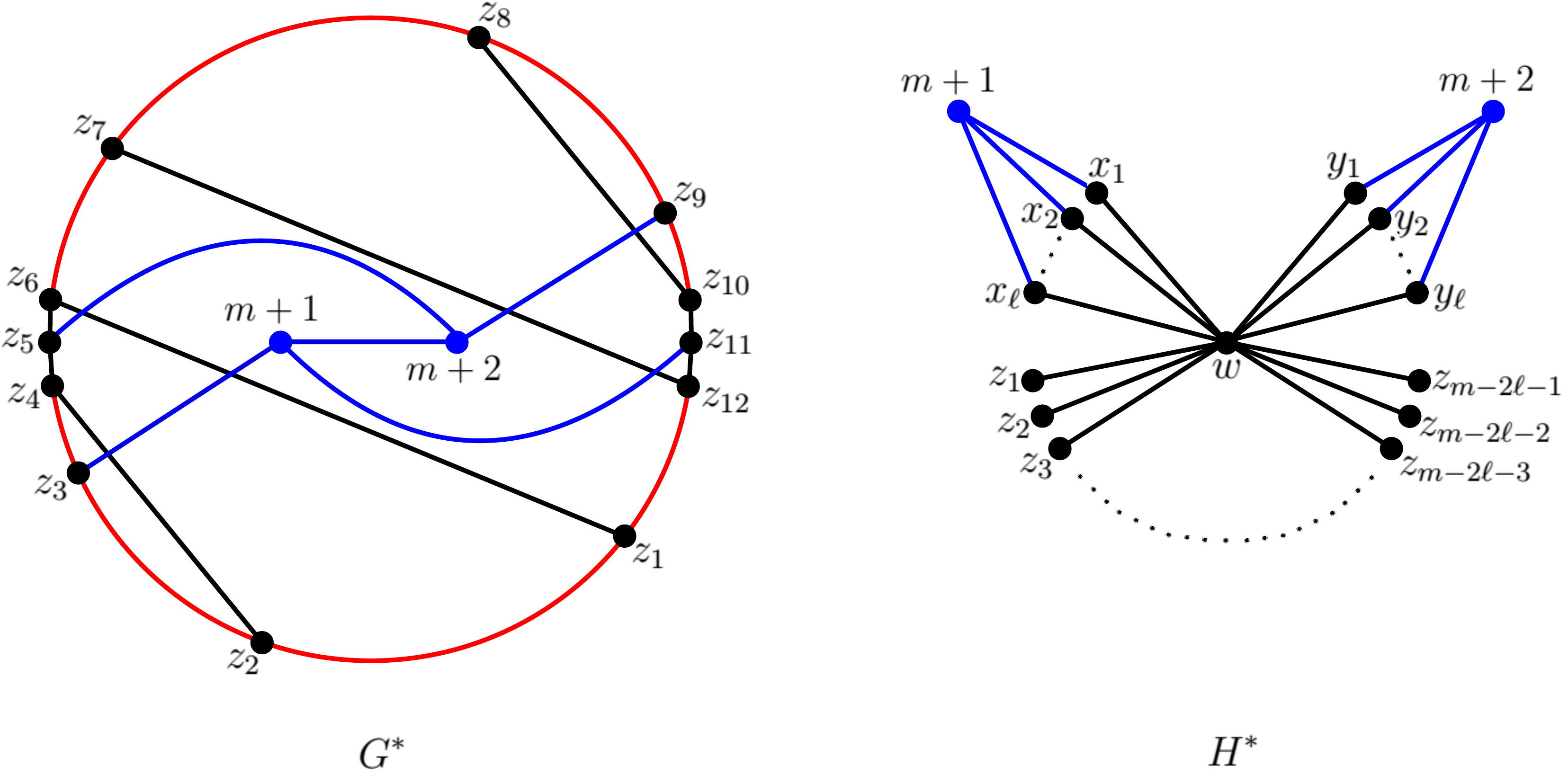}
\caption{Schematic diagrams of $G^*$ and $H^*$. The graphs $G^{**}$ and $H^{**}$ are obtained from $G^*$ and $H^*$, respectively, by removing the blue vertices and edges. In the diagram representing $G^*$, the red arcs on the large cycle are assumed to contain several vertices and edges, while each of the blue and black arcs represents a single edge. The vertices $w,x_1,\ldots,x_\ell,y_1,\ldots,y_\ell$ are not marked in the diagram of $G^*$; these vertices are placed in arbitrary positions so that the distance between any two of them, as well as the distance between any one of them and any one of the vertices $z_3,z_5,z_9,z_{11}$, is at least $m/(3\ell)$.}  
\label{Fig3}
\end{center}  
\end{figure}

\begin{lemma}\label{Lem2} 
Let $G^*$ and $H^*$ be the graphs constructed above. The vertices $m+1$ and $m+2$ are $(G^*,H^*)$-exchangeable from the identity bijection $\Id:[m+2]\to[m+2]$. 
\end{lemma}

\begin{proof}
Let $G^{**}=G^*\vert_{[m]}$ as above, and let $G^{***}=G^*\vert_{[m+2]\setminus\{z_5,z_{11}\}}$. It is easy to see that $G^{**}$ and $G^{***}$ are biconnected, not cycles, and not equal to the graph $\theta_0$ from Theorem~\ref{thm:wilson}. The smaller arc of the large cycle in $G^*$ connecting $z_2$ to $z_4$ has an even number of edges by construction, so adding the edge $\{z_2,z_4\}$ to that arc produces a cycle of odd length. It follows that $G^{**}$ and $G^{***}$ are not bipartite, so they are both Wilsonian (see Definition~\ref{def:wilsonian}). 

Let $\tau_1:[m+2]\to [m+2]$ be a bijection satisfying the following: 
\begin{itemize}
    \item The vertices $\tau_1^{-1}(x_1),\ldots,\tau_1^{-1}(x_\ell)$ appear in this order consecutively along the large cycle in $G^*$, with $\tau_1^{-1}(x_1)=z_3$ and $\tau_1^{-1}(x_\ell)=z_5$. 
    \item The vertices $\tau_1^{-1}(y_1),\ldots,\tau_1^{-1}(y_\ell)$ appear in this order consecutively along the large cycle in $G^*$, with $\tau_1^{-1}(y_1)=z_9$ and $\tau_1^{-1}(y_\ell)=z_{11}$.
    \item We have $\tau_1^{-1}(m+1)=m+1$ and $\tau_1^{-1}(m+2)=m+2$.
\end{itemize}
Since $H^{**}=H^*\vert_{[m]}$ is a star graph and $G^{**}$ is Wilsonian, we can use Wilson's Theorem~\ref{thm:wilson} to see that there exists a sequence $\Sigma_1$ of $(G^{**},H^{**})$-friendly swaps that transforms the bijection $\Id:[m]\to [m]$ into $\tau_1\vert_{[m]}$. We will view $\Sigma_1$ as a sequence of $(G^*,H^*)$-friendly swaps that does not involve $m+1$ or $m+2$; then $\Sigma_1$ transforms $\Id:[m+2]\to [m+2]$ into $\tau_1$. 

Applying the sequence \[\Sigma_2=(m+1)x_1,(m+1)x_2,\ldots,(m+1)x_\ell,(m+2)y_1,(m+2)y_2,\ldots,(m+2)y_\ell\] of $(G^*,H^*)$-friendly swaps transforms $\tau_1$ into a bijection $\tau_2:[m+2]\to[m+2]$ satisfying $\tau_2^{-1}(m+1)=z_5$ and $\tau_2^{-1}(m+2)=z_{11}$. Now let $\tau_3:[m+2]\to[m+2]$ be a bijection satisfying $\tau_3^{-1}(x_1)=m+2$, $\tau_3^{-1}(y_1)=m+1$, $\tau_3^{-1}(m+1)=z_5$, and $\tau_3^{-1}(m+2)=z_{11}$. Because $H^{**}=H^*\vert_{[m]}$ is a star graph and $G^{***}=G^*\vert_{[m+2]\setminus\{z_5,z_{11}\}}$ is Wilsonian, Wilson's Theorem~\ref{thm:wilson} guarantees the existence of a sequence $\Sigma_3$ of $(G^{***},H^{**})$-friendly swaps that transforms the bijection $\tau_2\vert_{[m+2]\setminus\{z_5,z_{11}\}}$ into $\tau_3\vert_{[m+2]\setminus\{z_5,z_{11}\}}$. We will view $\Sigma_3$ as a sequence of $(G^*,H^*)$-friendly swaps that does not involve $m+1$ or $m+2$; then $\Sigma_3$ transforms $\tau_2$ into $\tau_3$. We can now apply the $(G^*,H^*)$-friendly swaps $(m+1)x_1$ and $(m+2)y_1$ in order to transform $\tau_3$ into a bijection $\tau_4:[m+2]\to[m+2]$ satisfying $\tau_4^{-1}(m+1)=m+2$ and $\tau_4^{-1}(m+2)=m+1$. Finally, let $\tau_5:[m+2]\to[m+2]$ be the transposition that sends $m+1$ to $m+2$, sends $m+2$ to $m+1$, and sends $i$ to $i$ for all $i\in[m]$. We can once again use Wilson's Theorem~\ref{thm:wilson} with the star graph $H^{**}$ and the Wilsonian graph $G^{**}$ to see that there is a sequence $\Sigma_4$ of $(G^{**},H^{**})$-friendly swaps transforming $\tau_4\vert_{[m]}$ into the bijection $\Id:[m]\to [m]$. We can view $\Sigma_4$ as a sequence of $(G^*,H^*)$-friendly swaps that transforms $\tau_4$ into $\tau_5$. Putting this all together, we see that applying the sequence \[\Sigma_1,\Sigma_2,\Sigma_3,(m+1)x_1,(m+2)y_1,\Sigma_4\] of $(G^*,H^*)$-friendly swaps to $\Id:[m+2]\to[m+2]$ exchanges $m+1$ and $m+2$.  
\end{proof}

\begin{lemma}\label{Lem1}
Let $n$ be a large positive integer, and let $m$, $\ell$, $G^{**}$, and $H^{**}$ be as described above. Let $\Gamma=\{x_1,\ldots,x_\ell,y_1,\ldots,y_\ell,z_3,z_5,z_9,z_{11}\}$. Let $q_i=\left\lfloor pn/(3\ell)\right\rfloor$ for all $i\in \Gamma$, and let $q_i=\left\lfloor n/(2m)\right\rfloor$ for all $i\in[m]\setminus\Gamma$. Let $p=p(n)$ be a probability such that \[p\geq\frac{\exp(2(\log n)^{2/3})}{n^{1/2}}.\] Let $X$ and $Y$ be independently-chosen random graphs in $\mathcal G(n,p)$. If $n$ is sufficiently large, then the probability that the pair $(G^{**},H^{**})$ is $(q_1,\ldots,q_m)$-embeddable in $(X,Y)$ is at least $1-n^{-n/3}$.
\end{lemma}

\begin{proof}
Let $Q=q_1+\cdots+q_m$, and note that $n/3\leq Q\leq n$. For each set $J\subseteq [m]$, let $\beta(J)=|E(G^{**}\vert_J)|+|E(H^{**}\vert_J)|$ and $\gamma(J)=|J\cap \Gamma|$. 
The proof will follow from Lemma~\ref{Lem:Embeddable} if we can show that $p^{\beta(J)}\prod_{j\in J}q_j\geq 3\cdot 2^{m+1}n\log n$ for every $J\subseteq[m]$ satisfying $\beta(J)\geq 1$. Using the definitions of $q_1,\ldots,q_m$ (and ignoring floor symbols), we can rewrite this inequality as 
\begin{equation}\label{Eq11}
    p^{\beta(J)}\left(\frac{pn}{3\ell}\right)^{\gamma(J)}\left(\frac{n}{2m}\right)^{|J|-\gamma(J)}\geq 3\cdot 2^{m+1}n\log n.
\end{equation}

Let us first assume $J$ is such that $\beta(J)\geq 1$ and $w\not\in J$. The graph $H^{**}\vert_{J}$ has no edges, so the assumption that $\beta(J)\geq 1$ guarantees that there is an edge $\{t_1,t_2\}\in E(G^{**})$ such that $t_1,t_2\in J$. We constructed $G^{**}$ so that no two of the vertices in $\Gamma$ are adjacent, so one of the vertices in $\{t_1,t_2\}$, say $t_1$, is not in $\Gamma$. Now let $J'=(J\setminus\{t_1\})\cup\{w\}$, and observe that $|J'|=|J|$ and $\gamma(J')=\gamma(J)$. There are $|J|-1$ edges in $H^{**}\vert_{J'}$ and at most $|J|-1$ edges in $G^{**}\vert_J$ that are incident to $t_1$, so $\beta(J')\geq \beta(J)$. Consequently, \[p^{\beta(J)}\left(\frac{pn}{3\ell}\right)^{\gamma(J)}\left(\frac{n}{2m}\right)^{|J|-\gamma(J)}\geq p^{\beta(J')}\left(\frac{pn}{3\ell}\right)^{\gamma(J')}\left(\frac{n}{2m}\right)^{|J'|-\gamma(J')}.\] 

The previous paragraph demonstrates that in order to prove the lemma, it suffices to prove \eqref{Eq11} for all sets $J\subseteq[m]$ satisfying $\beta(J)\geq 1$ and $w\in J$. Assume $J$ satisfies these conditions, and observe that $|E(H^{**}\vert_J)|=|J|-1$ since $w\in J$. Let $\alpha(J)=|E(G^{**}\vert_J)|=\beta(J)-|J|+1$. With this notation, \[p^{\beta(J)}\left(\frac{pn}{3\ell}\right)^{\gamma(J)}\left(\frac{n}{2m}\right)^{|J|-\gamma(J)}= p^{\alpha(J)+\gamma(J)-|J|-1}\left(\frac{p^2n}{2m}\right)^{|J|}\left(\frac{2m}{3\ell}\right)^{\gamma(J)}\geq p^{\alpha(J)+\gamma(J)-|J|-1}\left(\frac{p^2n}{2m}\right)^{|J|}.\] Therefore, it suffices to prove that 
\begin{equation}\label{Eq12}
p^{\alpha(J)+\gamma(J)-|J|-1}\left(\frac{p^2n}{2m}\right)^{|J|}\geq 3\cdot 2^{m+1}n\log n
\end{equation}
for all sets $J\subseteq[m]$ satisfying $\beta(J)\geq 1$ and $w\in J$. To do this, we consider four cases. In what follows, let $c(J)$ be the number of connected components of $G^{**}\vert_J$. Let us also recall that $m=(\log n)^{2/3}$ and $\ell=\sqrt{m}/2=(\log n)^{1/3}/2$ (ignoring floor symbols). Furthermore, note that $p^2n/(2m)$ is certainly greater than $1$.

\noindent {\bf Case 1:} Suppose $|J|\geq m/6$. It follows from the construction of $G^{**}$ that $\alpha(J)\leq |J|+4$, so \[p^{\alpha(J)+\gamma(J)-|J|-1}\left(\frac{p^2n}{2m}\right)^{|J|}\geq p^{\gamma(J)+3}\left(\frac{p^2n}{2m}\right)^{|J|}\geq p^{|\Gamma|+3}\left(\frac{p^2n}{2m}\right)^{m/6}=p^{2\ell+7}\left(\frac{p^2n}{2m}\right)^{m/6}\] \[\geq n^{-(\ell+7/2)}\left(\frac{p^2n}{2m}\right)^{m/6} \geq n^{-(\ell+7/2)}\left(\frac{\exp(4(\log n)^{2/3})}{2m}\right)^{m/6}\] \[=\exp\left(-\left(\frac{1}{2}\left(\log n\right)^{1/3}+\frac{7}{2}\right)\log n\right)\left(\frac{\exp(4(\log n)^{2/3})}{2(\log n)^{2/3}}\right)^{(\log n)^{2/3}/6}\] \[=\exp\left(\frac{2}{3}(\log n)^{4/3}-\frac{1}{2}(\log n)^{4/3}+O(\log n)\right)=\exp\left(\frac{1}{6}(\log n)^{4/3}+O(\log n)\right),\] and this is certainly greater than $3\cdot 2^{m+1}n\log n$ if $n$ is sufficiently large. 

\noindent {\bf Case 2:} Suppose $|J|<m/6$ and $\gamma(J)\geq c(J)+2$. Because the graph $G^{**}$ has girth at least $m/6$, the induced subgraph $G^{**}\vert_J$ must be a forest. This implies that $\alpha(J)=|J|-c(J)$. Now recall that for any distinct $s_1,s_2\in\Gamma\cup\{w\}$, if $\{s_1,s_2\}\neq\{z_3,z_5\}$ and $\{s_1,s_2\}\neq\{z_9,z_{11}\}$, then the distance between $s_1$ and $s_2$ in $G^{**}$ is at least $m/(3\ell)$. Since $w\in J$, it is straightforward to check that $|J|\geq (\gamma(J)-c(J)-1)m/(3\ell)$; indeed, if a connected component of $G^{**}|_J$ contains $k$ elements of $(\Gamma \cup \{w\}) \setminus \{z_3, z_9\}$, then this connected component must contain at least $(k-1)m/(3\ell)$ vertices. Therefore, \[p^{\alpha(J)+\gamma(J)-|J|-1}\left(\frac{p^2n}{2m}\right)^{|J|}=p^{\gamma(J)-c(J)-1}\left(\frac{p^2n}{2m}\right)^{|J|}\geq p^{\gamma(J)-c(J)-1}\left(\frac{p^2n}{2m}\right)^{(\gamma(J)-c(J)-1)m/(3\ell)}\] \[\geq\left(\frac{(p^2n)^{m/(3\ell)}}{(2m)^{m/(3\ell)}}n^{-1/2}\right)^{\gamma(J)-c(J)-1}= \left((p^2n)^{m/(3\ell)}n^{-1/2+o(1)}\right)^{\gamma(J)-c(J)-1}\]
\[=\left(\left(\exp\left(4\left(\log n\right)^{2/3}\right)\right)^{(2/3)(\log n)^{1/3}}n^{-1/2+o(1)}\right)^{\gamma(J)-c(J)-1}\] \[=\left(n^{8/3-1/2+o(1)}\right)^{\gamma(J)-c(J)-1}\geq n^{13/6+o(1)},\] and this is greater than $3\cdot 2^{m+1}n\log n$ if $n$ is sufficiently large. 

\noindent {\bf Case 3:} Suppose $|J|<m/6$ and $c(J)\leq\gamma(J)\leq c(J)+1$. As in the previous case, the lower bound on the girth of $G^{**}$ forces $G^{**}\vert_J$ to be a forest, so $\alpha(J)=|J|-c(J)$. The number of elements of $\Gamma\cup\{w\}$ that are in $J$ is $\gamma(J)+1$, which is at least $c(J)+1$. This means that some connected component of $G^{**}\vert_J$ contains at least $2$ elements of $\Gamma\cup\{w\}$. The minimum distance in $G^{**}$ between any two elements of $\Gamma\cup\{w\}$ is $\ell-1$, so $|J|\geq \ell$. It follows that \[p^{\alpha(J)+\gamma(J)-|J|-1}\left(\frac{p^2n}{2m}\right)^{|J|}\geq p^{\gamma(J)-c(J)-1}\left(\frac{p^2n}{2m}\right)^{\ell}\geq\left(\frac{p^2n}{2m}\right)^{\ell}=\frac{\exp\left(4\ell\left(\log n\right)^{2/3}\right)}{(2m)^\ell}\] \[=\frac{\exp\left(2\left(\log n\right)\right)}{(2m)^\ell}=n^{2+o(1)},\] and this is greater than $3\cdot 2^{m+1}n\log n$ if $n$ is sufficiently large. 

\noindent {\bf Case 4:} Suppose $|J|<m/6$ and $\gamma(J)\leq c(J)-1$. As in the previous two cases, the assumption that $|J|<m/6$ forces $G^{**}\vert_J$ to be a forest so that $\alpha(J)=|J|-c(J)$. Furthermore, $|J|\geq 2$ because $\beta(J)\geq 1$. Consequently, \[p^{\alpha(J)+\gamma(J)-|J|-1}\left(\frac{p^2n}{2m}\right)^{|J|}=p^{\gamma(J)-c(J)-1}\left(\frac{p^2n}{2m}\right)^{|J|}\geq p^{-2}\left(\frac{p^2n}{2m}\right)^{2}=\frac{p^2n^2}{(2m)^2}\] \[=\frac{\exp\left(4(\log n)^{2/3}\right)n}{(2m)^2}=\frac{e^{4m}n}{(2m)^2},\] and this is greater than $3\cdot 2^{m+1}n\log n$ if $n$ is sufficiently large. 
\end{proof}

We can now complete the proof of Theorem~\ref{thm:random}. 

\begin{proposition}
Let $X$ and $Y$ be independently-chosen random graphs in $\mathcal G(n,p)$. If \[p\geq\frac{\exp(2(\log n)^{2/3})}{n^{1/2}},\] then $\FS(X,Y)$ is connected with high probability.  
\end{proposition}

\begin{proof}
Let $n$ be a large positive integer, and let $m$, $\ell$, $G^*$, $G^{**}$, $H^*$, and $H^{**}$ be as described above. Let $\Gamma$ and $q_1,\ldots,q_m$ be as in the statement of Lemma~\ref{Lem1}. We may assume that the pair $(G^{**},H^{**})$ is $(q_1,\ldots,q_m)$-embeddable in $(X,Y)$ since Lemma~\ref{Lem1} tells us that this happens with high probability. Because $p$ is much larger than $\log n/n$, it is well known that with high probability, $X$ and $Y$ are connected and the degrees of all vertices in $X$ and $Y$ are $pn(1+o(1))$; hence, we may assume $X$ and $Y$ have these properties. 

Choose vertices $u,v\in V(Y)$ and a bijection $\sigma:V(X)\to V(Y)$ such that $\{\sigma^{-1}(u),\sigma^{-1}(v)\}\in E(X)$. Let us choose pairwise disjoint subsets $V_1,\ldots,V_m$ of $V(Y)\setminus\{u,v\}$ such that 
\begin{itemize}
    \item $|V_i|=q_i$ for all $i\in[m]$;
    \item $V_{x_1},\ldots,V_{x_\ell}$ are all contained in the neighborhood of $u$ in $Y$;
    \item $V_{y_1},\ldots,V_{y_\ell}$ are all contained in the neighborhood of $v$ in $Y$;
    \item $V_{z_3}$ and $V_{z_{11}}$ are contained in the neighborhood of $\sigma^{-1}(u)$ in $X$;
    \item $V_{z_5}$ and $V_{z_9}$ are contained in the neighborhood of $\sigma^{-1}(v)$ in $X$. 
\end{itemize}    
Note that such a choice is possible because $q_i=\left\lfloor pn/(3\ell)\right\rfloor$ for all $i\in\Gamma$. Because the pair $(G^{**},H^{**})$ is $(q_1,\ldots,q_m)$-embeddable in $(X,Y)$, it must be the case that $(G^{**},H^{**})$ is embeddable in $(X,Y)$ with respect to the sets $V_1,\ldots,V_m$ and the bijection $\sigma$. This means that there exist vertices $v_i\in V_i$ for all $i\in[m]$ such that for all $i,j\in [m]$, we have 
\[\{i,j\}\in E(H^{**})\implies\{v_i,v_j\}\in E(Y)\quad\text{and}\quad\{i,j\}\in E(G^{**})\implies\{\sigma^{-1}(v_i),\sigma^{-1}(v_j)\}\in E(X).\] 

Define a map $\psi:V(H^*)\to V(Y)$ by $\psi(m+1)=u$, $\psi(m+2)=v$, and $\psi(i)=v_i$ for all $i\in [m]$. Define $\varphi: V(G^*)\to V(X)$ by $\varphi=\varphi\circ\Id=\sigma^{-1}\circ\psi$. Because the vertices $v_1,\ldots,v_m,u,v$ are distinct, the maps $\psi$ and $\varphi$ are injective. It is immediate from our construction that $\psi$ is a graph embedding of $H^*$ into $Y$ that sends $m+1$ to $u$ and sends $m+2$ to $v$. Similarly, $\varphi$ is a graph embedding of $G^*$ into $X$ that sends $m+1$ to $\sigma^{-1}(u)$ and sends $m+2$ to $\sigma^{-1}(v)$. Lemma~\ref{Lem2} tells us that there is a sequence $\Sigma$ of $(G^*,H^*)$-friendly swaps that we can apply to the identity bijection $\Id:[m+2]\to [m+2]$ in order to exchange $m+1$ and $m+2$. Using the graph embeddings $\psi$ and $\varphi=\sigma^{-1}\circ\psi$, we can transfer $\Sigma$ to a sequence of $(X,Y)$-friendly swaps that we can apply to $\sigma$ in order to exchange $u$ and $v$. This proves that $u$ and $v$ are $(X,Y)$-exchangeable from $\sigma$; as $u$, $v$, and $\sigma$ were arbitrary (subject to the condition $\{\sigma^{-1}(u),\sigma^{-1}(v)\}\in E(X)$), it follows from Lemma~\ref{lem:exchangeable-K_n} that $\FS(X,Y)$ is connected. 
\end{proof}

\section{Random Bipartite Graphs}\label{Sec:RandomBipartite}

\subsection{Disconnectedness with high probability}

The following proposition implies the first statement in Theorem~\ref{thm:random}. 

\begin{proposition}
Fix some small $\varepsilon>0$, and let $X$ and $Y$ be independently-chosen random bipartite graphs in $\mathcal G(K_{r,r},p)$. If \[p=p(r)\leq\frac{1-\varepsilon}{r^{1/2}},\] then the friends-and-strangers graph $\FS(X,Y)$ has an isolated vertex with high probability. 
\end{proposition}

\begin{proof}
The proof is essentially the same as that of Proposition~\ref{Prop:NoIsolated}. It follows from \cite{Sauer} that if $G$ and $H$ are $n$-vertex graphs with maximum degrees $\Delta(G)$ and $\Delta(H)$ satisfying $2\Delta(G)\Delta(H)<n$, then $\FS(G,H)$ has an isolated vertex. Now consider the random graphs $X$ and $Y$, each of which has $2r$ vertices. It suffices to assume $p=(1-\varepsilon)/r^{1/2}$; in this case, it is well known that $\Delta(X)=pr(1+o(1))$ and $\Delta(Y)=pr(1+o(1))$ with high probability. Consequently, $2\Delta(X)\Delta(Y)=2p^2r^2(1+o(1))\leq 2\left(1-\varepsilon\right)^2(1+o(1))r<2r$ with high probability. 
\end{proof}

\subsection{Connectedness with high probability}

We now turn to the second statement in Theorem~\ref{thm:randombipartite}.  The argument requires a modification of the techniques from Section~\ref{Sec:random-upper}, so we begin by looking at embeddability in bipartite graphs.  Let $G$ and $H$ be bipartite graphs on $m$ vertices with bipartitions $\{A_G, B_G\}$ and $\{A_H,B_H\}$, respectively.  Let $X$ and $Y$ be bipartite graphs on $2r$ vertices with bipartitions $\{A_X, B_X\}$ and $\{A_Y,B_Y\}$, respectively, where $|A_X|=|B_X|=|A_Y|=|B_Y|=r$.  Let $\sigma:V(X) \to V(Y)$ be a bijection.  We say a list $V_1, \ldots, V_{m}$ of pairwise disjoint subsets of $V(Y)$ is \emph{admissible for $\sigma$ with respect to $(G,H)$} if the following conditions hold:
\begin{itemize}
\item The subsets $V_i$ for $i \in A_H$ are all contained in one of the partite sets of $Y$, and the subsets $V_i$ for $i \in B_H$ are all contained in the other partite set of $Y$.
\item The subsets $\sigma^{-1}(V_i)$ for $i \in A_G$ are all contained in one of the partite sets of $X$, and the subsets $\sigma^{-1}(V_i)$ for $i \in B_G$ are all contained in the other partite set of $X$.
\end{itemize}
These ``correlation'' conditions on the $V_i$'s will prevent parity obstructions when we try to find embeddings.

Recall that a pair of graphs $(G,H)$ is said to be embeddable in $(X,Y)$ with respect to the sets $V_1, \ldots, V_m$ and the bijection $\sigma$ if there exist vertices $v_i\in V_i$ for all $i\in[m]$ such that for all $i,j\in [m]$, we have 
\[\{i,j\}\in E(H)\implies\{v_i,v_j\}\in E(Y)\quad\text{and}\quad\{i,j\}\in E(G)\implies\{\sigma^{-1}(v_i),\sigma^{-1}(v_j)\}\in E(X).\]
Now, suppose $q_1, \ldots, q_m$ are integers satisfying $q_1+\cdots+q_m\leq 2r$.  We say that the pair $(G,H)$ is \emph{$(q_1, \ldots, q_m)$-bipartite embeddable in $(X,Y)$} if the pair $(G,H)$ is embeddable in $(X,Y)$ with respect to the sets $V_1, \ldots, V_m$ and the bijection $\sigma$ for every bijection $\sigma$ and every list $V_1, \ldots, V_m$ that is admissible for $\sigma$ with respect to $(G,H)$ and satisfies $|V_i|=q_i$ for all $i \in [m]$.  (Of course, one could extend this definition to the case where the parts of the vertex bipartitions of $X$ and $Y$ do not all have size $r$, but we do not state this asymmetric version because it is more complicated than what we will need.)  

\begin{lemma}\label{lem:bipartite-embeddability}
Let $G$ and $H$ be bipartite graphs on the vertex set $[m]$ with bipartitions $\{A_G, B_G\}$ and $\{A_H,B_H\}$, respectively.  Let $r, q_1, \ldots, q_m$ be positive integers such that $Q:=q_1+\cdots+q_m\leq 2r$. For every set $J\subseteq [m]$, let $\beta(J)=|E(G\vert_J)|+|E(H\vert_J)|$. Choose $0\leq p\leq 1$, and let $X$ and $Y$ be independently-chosen random graphs in $\mathcal G(K_{r,r},p)$. If for every set $J\subseteq [m]$ satisfying $\beta(J)\geq 1$ we have \[p^{\beta(J)}\prod_{j\in J}q_j\geq 3\cdot 2^{m+1}Q\log (2r),\] then the probability that the pair $(G,H)$ is $(q_1,\ldots,q_m)$-bipartite-embeddable in $(X,Y)$ is at least $1-(2r)^{-Q}$.
\end{lemma}
\begin{proof}
Let us view $K_{r,r}$ as an edge-subgraph of $K_{2r}$. We can choose the graphs $X$ and $Y$ by first choosing independent graphs $\widetilde X$ and $\widetilde Y$ in $\mathcal G(2r,p)$ and then deleting edges that connect two vertices within the same partite set of $K_{r,r}$. If the pair $(G,H)$ is $(q_1,\ldots,q_m)$-embeddable in $(\widetilde X,\widetilde Y)$, then it is $(q_1,\ldots,q_m)$-bipartite-embeddable in $(X,Y)$. Lemma~\ref{Lem:Embeddable} tells us that this happens with probability at least $1-(2r)^{-Q}$. 
\end{proof}

Recall that we proved the second statement in Theorem~\ref{thm:random} by applying Lemma~\ref{Lem:Embeddable} with two graphs $G^{**}$ and $H^{**}$, each with $m=\left\lfloor(\log n)^{2/3}\right\rfloor$ vertices. In order to use these specific graphs, we repeatedly made use of Wilson's Theorem~\ref{thm:wilson}. Unfortunately, the known bipartite analogue of Theorem~\ref{thm:wilson} (see~\cite{wilson}) is not sufficiently robust, and we do not know of a suitable substitute, so we will not apply Lemma~\ref{lem:bipartite-embeddability} with $m\to\infty$. Rather, we will apply it with four particular pairs of graphs, each with $m=8$; this is why the two bounds in Theorem~\ref{thm:randombipartite} differ by a multiplicative factor of $r^{1/5+o(1)}$ while the bounds in Theorem~\ref{thm:random} differ by a multiplicative factor of $n^{o(1)}$.

\begin{figure}[ht]
\begin{center}
\includegraphics[height=11.3cm]{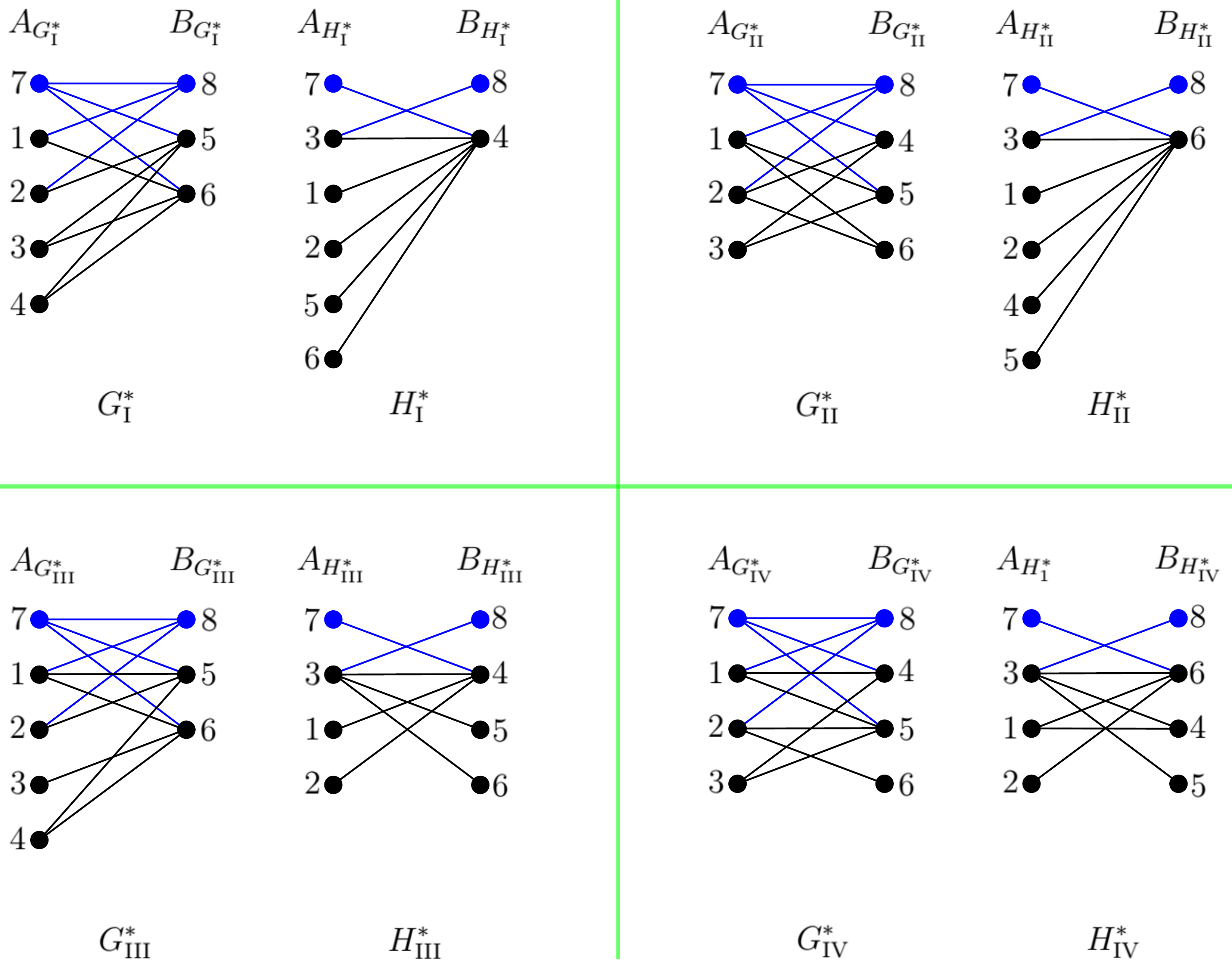}
\caption{Diagrams showing the graphs $G_{\text{R}}^*$ and $H_{\text{R}}^*$ for $\text{R}\in\{\text{I},\text{II},\text{III},\text{IV}\}$. The graphs $G_{\text{R}}^{**}$ and $H_{\text{R}}^{**}$ are obtained from $G_{\text{R}}^*$ and $H_{\text{R}}^*$, respectively, by removing the blue vertices and edges.}
\label{Fig4}
\end{center}  
\end{figure}

We now introduce the four pairs of graphs that we will need.  Figure~\ref{Fig4} shows the pairs of graphs $(G^*_{\text{I}}, H^*_{\text{I}})$, $(G^*_{\text{II}}, H^*_{\text{II}})$, $(G^*_{\text{III}}, H^*_{\text{III}})$, and $(G^*_{\text{IV}}, H^*_{\text{IV}})$, where each graph has the vertex set $\{1,\ldots,8\}$. For each $\text{R} \in \{\text{I}, \text{II}, \text{III}, \text{IV}\}$, let $(G^{**}_{\text{R}}, H^{**}_{\text{R}})$ be obtained from $(G^{*}_{\text{R}}, H^{*}_{\text{R}})$ by deleting the vertices $7$ and $8$ from each graph.  Thus, $G^{**}_{\text{R}}$ and $H^{**}_{\text{R}}$ are graphs on the vertex set $\{1,\ldots,6\}$.  
\begin{lemma}\label{lem:bipartite-exchangeable}
For every $\operatorname{R} \in \{\operatorname{I}, \operatorname{II}, \operatorname{III}, \operatorname{IV}\}$, the vertices $7$ and $8$ are $(G^*_{\operatorname{R}}, H^*_{\operatorname{R}})$-exchangeable from the identity bijection $\Id:\{1,\ldots,8\}\to\{1,\ldots,8\}$.
\end{lemma}
\begin{proof}
In each case, we will simply state the sequence of $(G^*_{\operatorname{R}}, H^*_{\operatorname{R}})$-friendly swaps (which we found using computer assistance) that can be applied to the identity bijection in order to exchange $7$ and $8$. For $(G^*_{\text{I}}, H^*_{\text{I}})$, the sequence of swaps is
\[46, 34, 45, 47, 34, 38, 46, 47, 24, 34, 14, 46, 45, 24, 34, 14, 46,\] \[45, 24, 34, 47, 24, 45, 46, 14, 47, 24, 45, 46, 14, 47, 24, 45.\]
For $(G^*_{\text{II}}, H^*_{\text{II}})$, the sequence of swaps is
\[16, 56, 67, 46, 36, 67, 46, 36, 26, 16, 56, 38, 36, 16, 26, 67, 46, 36, 56, 26, 67, 46, 36, 56, 16, 67, 26.\]
For $(G^*_{\text{III}}, H^*_{\text{III}})$, the sequence of swaps is
\[36, 34, 47, 35, 34, 38, 14, 34, 24, 14, 34, 35, 47, 34, 38, 35, 34, 14, 24, 47, 14, 34, 35, 38, 36.\]
Finally, for $(G^*_{\text{IV}}, H^*_{\text{IV}})$, the sequence of swaps is
\[26, 35, 14, 34, 38, 36, 34, 67, 38, 16, 36, 35, 38, 34, 14, 16, 36, 38, 26, 35, 34, 38, 35. \qedhere\]
\end{proof}

\begin{lemma}\label{lem:bipartite-p-optimization}
Let $r$ be a large positive integer, and let the pairs $(G^{**}_{\operatorname{R}}, H^{**}_{\operatorname{R}})$ be as above.  Let $p=p(r)$ be a probability that satisfies
$$p\geq \frac{5 (\log r)^{1/10}}{r^{3/10}}.$$
Let $X$ and $Y$ be independently-chosen random graphs in $\mathcal{G}(K_{r,r},p)$.  Let $q=\left\lfloor pr/13\right\rfloor$. If $r$ is sufficiently large, then the probability that every pair $(G^{**}_{\operatorname{R}}, H^{**}_{\operatorname{R}})$ (for $\operatorname{R} \in \{\operatorname{I}, \operatorname{II}, \operatorname{III}, \operatorname{IV}\}$) is $(q,q,q,q,q,q)$-bipartite-embeddable in $(X,Y)$ is at least $1-4(2r)^{-6q}$.
\end{lemma}
\begin{proof}
Fix $\text{R}\in\{\text{I},\text{II},\text{III},\text{IV}\}$. Preserve the notation from Lemma~\ref{lem:bipartite-embeddability} with $m=6$, with $q_1=\cdots=q_6=q$, and with $(G^{**}_{\text{R}}, H^{**}_{\text{R}})$ playing the role of $(G,H)$. Note that $Q=6q$. It is 
straightforward (yet somewhat tedious) to verify that if $r$ is sufficiently large, then \[p^{\beta(J)}\prod_{j\in J}q_j\geq 3\cdot2^{m+1}Q\log(2r)\] for every set $J\subseteq\{1,\ldots,6\}$ satisfying $\beta(J)\geq 1$. Indeed, verifying this inequality for a fixed $J$ is easy, and there are at most $64$ possible choices of $J$. By 
Lemma~\ref{lem:bipartite-embeddability}, the probability that $(G^{**}_{\text{R}}, H^{**}_{\text{R}})$ is $(q,q,q,q,q,q)$-bipartite-embeddable in $(X,Y)$ is at least $1-(2r)^{-6q}$. The desired result now follows from taking a union bound over the $4$ choices for $\text{R}$. 
\end{proof}

We are now in a position to prove the second statement of Theorem~\ref{thm:randombipartite}.

\begin{proposition}\label{prop:bipartite-random}
Let $X$ and $Y$ be independently-chosen random graphs in $\mathcal{G}(K_{r,r},p)$, where
\[p\geq \frac{5 (\log r)^{1/10}}{r^{3/10}}.\] 
The following holds with high probability: For every edge-supergraph $\widetilde X$ of $X$, every bijection $\sigma:V(X)\to V(Y)$, and all vertices $u$ and $v$ in different partite sets of $Y$ such that $\{\sigma^{-1}(u),\sigma^{-1}(v)\}\in E(\widetilde X)$, we have that the vertices $u$ and $v$ are $(\widetilde X,Y)$-exchangeable from $\sigma$.
\end{proposition}

Before we prove this proposition, let us see how it implies the second statement in Theorem~\ref{thm:randombipartite}. Choose $X$ and $Y$ independently from $\mathcal G(K_{r,r},p)$, where $p\geq 5 (\log r)^{1/10}/r^{3/10}$.  Proposition~\ref{prop:bipartite-random} is clearly symmetric in $X$ and $Y$, so we can apply it (with the roles of $X$ and $Y$ switched and $\widetilde X=Y$) in conjunction with Proposition~\ref{lem:exchangeable} to see that, with high probability, $\FS(Y,X)$ has the same number of connected components as $\FS(Y,K_{r,r})$. On the other hand, we can also apply Proposition~\ref{prop:bipartite-random} directly (with $\widetilde X=K_{r,r}$) in conjunction with Proposition~\ref{lem:exchangeable} to see that, with high probability, $\FS(K_{r,r},Y)$ has the same number of connected components as $\FS(K_{r,r},K_{r,r})$. Since $\FS(X,Y)\cong\FS(Y,X)$ and  $\FS(Y,K_{r,r})\cong\FS(K_{r,r},Y)$, it follows that $\FS(X,Y)$ has the same number of connected components as $\FS(K_{r,r},K_{r,r})$ with high probability. This number of connected components is $2$ by Proposition~\ref{prop:complete-bipartite}.  

\begin{proof}[Proof of Proposition~\ref{prop:bipartite-random}]
As usual, let $X$ and $Y$ have vertex bipartitions $\{A_X,B_X\}$ and $\{A_Y,B_Y\}$, respectively. Let $q=\left\lfloor pr/13\right\rfloor$. With high probability, all of the vertices in $X$ and $Y$ have degrees $pr(1+o(1))$; henceforth, we will assume that this is the case. It follows from Lemma~\ref{lem:bipartite-p-optimization} that, with high probability, every pair $(G_{\text{R}}^{**},H_{\text{R}}^{**})$ for $\text R\in\{\text I,\text{II},\text{III},\text{IV}\}$ is $(q,q,q,q,q,q)$-bipartite-embeddable in $(X,Y)$; henceforth, we will assume that this is the case as well.  Suppose $X$ is an edge-subgraph of a graph $\widetilde X$. Fix a bijection $\sigma:V(X)\to V(Y)$ and vertices $u,v$ in different partite sets of $Y$ such that $\{\sigma^{-1}(u),\sigma^{-1}(v)\}\in E(\widetilde X)$. For convenience, let $u'=\sigma^{-1}(u)$ and $v'=\sigma^{-1}(v)$. Without loss of generality, suppose that $u' \in A_X$, $v' \in B_X$, $u \in A_Y$, and $v \in B_Y$.  We now distinguish four cases, depending on how the neighborhoods of $u'$ and $v'$ (respectively, $u$ and $v$) correlate with the bipartition $\{A_Y, B_Y\}$ (respectively, $\{A_X,B_X\}$) after an application of $\sigma$ (respectively, $\sigma^{-1}$).  In what follows, the neighborhood of a vertex in $V(X)$ is taken with respect to $X$ (not $\widetilde X$). Say that $\sigma$ \emph{majority-maps} $N(u')$ (respectively, $N(v')$) into $A_Y$ if $A_Y$ contains at least half of the elements of $\sigma(N(u'))$ (respectively, $N(v')$), and say that $\sigma$ majority-maps $N(u')$ (respectively, $N(v')$) into $B_Y$ otherwise. We give a similar definition of $\sigma^{-1}$ majority-mapping $N(u)$ or $N(v)$ into $A_X$ or $B_X$. Our four cases are the following:
\begin{enumerate}[I.]
\item $\sigma$ majority-maps $N(u')$ and $N(v')$ into the same part of the bipartition of $Y$, and $\sigma^{-1}$ majority-maps $N(u)$ and $N(v)$ into the same part of the bipartition of $X$.
\item $\sigma$ majority-maps $N(u')$ and $N(v')$ into the same part of the bipartition of $Y$, and $\sigma^{-1}$ majority-maps $N(u)$ and $N(v)$ into different parts of the bipartition of $X$.
\item $\sigma$ majority-maps $N(u')$ and $N(v')$ into different parts of the bipartition of $Y$, and $\sigma^{-1}$ majority-maps $N(u)$ and $N(v)$ into the same part of the bipartition of $X$.
\item $\sigma$ majority-maps $N(u')$ and $N(v')$ into different parts of the bipartition of $Y$, and $\sigma^{-1}$ majority-maps $N(u)$ and $N(v)$ into different parts of the bipartition of $X$.
\end{enumerate}

We will give a detailed explanation of the proof in Case I, but we only sketch the other three cases because they are very similar. 

\noindent \textbf{Case I:} Without loss of generality, we may assume that $\sigma$ majority-maps both $N(u')$ and $N(v')$ into $A_Y$ (otherwise, we can switch the roles of $u$ and $v$ and switch the roles of $u'$ and $v'$).  Note that $u$ and $v$ are $(\widetilde X,Y)$-exchangeable from $\sigma$ if and only they are $(\widetilde X,Y)$-exchangeable from $\sigma \circ (u'\,\,v')$; thus, by possibly replacing $\sigma$ with $\sigma \circ (u'\,\,v')$ and switching the names of $A_X$ and $B_X$, we may also assume that $\sigma^{-1}$ majority-maps both $N(u)$ and $N(v)$ into $A_X$.  Each of the sets $N(u')$, $N(v')$, $N(u)$, $N(v)$ has size $pr(1+o(1))$, so we can find pairwise disjoint sets $V_1, \ldots, V_6$ in $V(Y)\setminus\{u,v\}$, each of size $q=\left\lfloor pr/13\right\rfloor$, such that the following hold:
\begin{itemize}
\item $\sigma^{-1}(V_5), \sigma^{-1}(V_6) \subseteq N(u')$ (which implies that $\sigma^{-1}(V_5), \sigma^{-1}(V_6) \subseteq B_X$) and $V_5, V_6 \subseteq A_Y$.
\item $\sigma^{-1}(V_1), \sigma^{-1}(V_2) \subseteq N(v')$ (which implies that $\sigma^{-1}(V_1), \sigma^{-1}(V_2) \subseteq A_X$) and $V_1, V_2 \subseteq A_Y$.
\item $V_4 \subseteq N(u)$ (which implies that $V_4 \subseteq B_Y$) and $\sigma^{-1}(V_4) \subseteq A_X$.
\item $V_3 \subseteq N(v)$ (which implies that $V_3 \subseteq A_Y$) and $\sigma^{-1}(V_3) \subseteq A_X$.
\end{itemize}
Note that the list $V_1, \ldots, V_6$ is admissible for $\sigma$ with respect to $(G_{\text{I}}^{**},H_{\text{I}}^{**})$. Since the pair $(G_{\text{I}}^{**},H_{\text{I}}^{**})$ is $(q,q,q,q,q,q)$-bipartite-embeddable in $(X,Y)$, it must be embeddable in $(X,Y)$ with respect to the sets $V_1\ldots,V_6$ and the bijection $\sigma$. This means that there exist vertices $v_i \in V_i$ for all $i\in\{1,\ldots,6\}$ such that for all $i,j\in [m]$, we have 
\[\{i,j\}\in E(H_{\text I}^{**})\implies\{v_i,v_j\}\in E(Y)\quad\text{and}\quad\{i,j\}\in E(G_{\text I}^{**})\implies\{\sigma^{-1}(v_i),\sigma^{-1}(v_j)\}\in E(X).\] Let $v_i'=\sigma^{-1}(v_i)$ for all $i\in\{1,\ldots,6\}$. We have \[\{u', v'_5\}, \{u', v'_6\}, \{v', v'_1\}, \{v', v'_2\} \in E(X) \quad \text{and} \quad \{u, v_4\}, \{v, v_3\} \in E(Y)\] by construction.  

Define a map $\psi:V(H_{\text I}^*)\to V(Y)$ by $\psi(7)=u$, $\psi(8)=v$, and $\psi(i)=v_i$ for all $i\in\{1,\ldots,6\}$. Define $\varphi:V(G_{\text I}^*)\to V(\widetilde X)$ by $\varphi=\varphi\circ\Id=\sigma^{-1}\circ\psi$. Note that $\psi$ and $\varphi$ are injective. By the discussion in the previous paragraph, $\psi$ is a graph embedding of $H_{\text I}^*$ into $Y$ that sends $7$ to $u$ and sends $8$ to $v$. Similarly, $\varphi$ is a graph embedding of $G_\text{I}^*$ into $\widetilde X$ that sends $7$ to $u'$ and sends $8$ to $v'$.  Lemma~\ref{lem:bipartite-exchangeable} tells us that there is a sequence $\Sigma_{\text I}$ of $(G_{\text I}^*,H_{\text I}^*)$-friendly swaps that we can apply to the identity bijection $\Id:\{1,\ldots,8\}\to\{1,\ldots,8\}$ in order to exchange $7$ and $8$; using the graph embeddings $\psi$ and $\varphi=\sigma^{-1}\circ\psi$, we can transfer $\Sigma_{\text I}$ to a sequence of $(\widetilde X,Y)$-friendly swaps that we can apply to $\sigma$ in order to exchange $u$ and $v$.

\noindent \textbf{Case II:} As in Case I, we may assume without loss of generality that $\sigma$ majority-maps both $N(u')$ and $N(v')$ into $A_Y$. By possibly replacing $\sigma$ with $\sigma \circ (u'\,\,v')$ and switching the names of $A_X$ and $B_X$, we may also assume that $\sigma^{-1}$ majority-maps $N(u)$ into $B_X$ and majority-maps $N(v)$ into $A_X$.  We can produce pairwise disjoint sets $V_1, \ldots, V_6$ in $V(Y)\setminus\{u,v\}$, each of size $q$, such that the following hold:
\begin{itemize}
\item $\sigma^{-1}(V_4), \sigma^{-1}(V_5) \subseteq N(u')$ and $V_4, V_5 \subseteq A_Y$.
\item $\sigma^{-1}(V_1), \sigma^{-1}(V_2) \subseteq N(v')$ and $V_1, V_2 \subseteq A_Y$.
\item $V_3 \subseteq N(u)$ and $\sigma^{-1}(V_3) \subseteq A_X$.
\item $V_6 \subseteq N(v)$ and $\sigma^{-1}(V_3) \subseteq B_X$.
\end{itemize}
The list $V_1, \ldots, V_6$ is admissible for $\sigma$ with respect to $(G_{\text{II}}^{**},H_{\text{II}}^{**})$. The same argument as in Case~I, except with $G_{\text{I}}^*,H_{\text{I}}^*,G_{\text{I}}^{**},H_{\text{I}}^{**}$ replaced by $G_{\text{II}}^*,H_{\text{II}}^*,G_{\text{II}}^{**},H_{\text{II}}^{**}$, shows that $u$ and $v$ are $(\widetilde X,Y)$-exchangeable from $\sigma$.

\noindent \textbf{Cases III and IV:} We omit the details of these cases, which are entirely analogous to Cases~I and II.  In Case R (for $\text{R}\in\{\text{III},\text{IV}\}$), we find the necessary pairwise disjoint sets $V_1, \ldots, V_6$ in $V(Y)\setminus\{u,v\}$, each of size $q$, such that the list $V_1,\ldots,V_6$ is admissible for $\sigma$ with respect to $(G_{\text{R}}^{**},H_{\text{R}}^{**})$. Repeating the same argument as before shows that $u$ and $v$ are $(\widetilde X,Y)$-exchangeable from $\sigma$.
\end{proof}

\section{Graphs with Large Minimum Degree}\label{Sec:MinDegree}

The purpose of this section is to prove Theorem~\ref{Thm:bipartitemindegree}, which gives bounds for $d_n$. Recall that this is the smallest nonnegative integer such that any two $n$-vertex graphs $X$ and $Y$ with minimum degrees at least $d_n$ must have a connected friends-and-strangers graph $\FS(X,Y)$. 

\subsection{Lower bound}

\begin{proposition}
We have $d_n\geq \dfrac{3}{5}n-2$. 
\end{proposition}

\begin{proof}
The proposition is trivial if $n\leq 4$, so we may assume $n\geq 5$.  We exhibt the bound by constructing an explicit family of examples. Partition the set $[n]$ into $5$ subsets $A_1,\ldots,A_5$, each of size $\left\lfloor n/5\right\rfloor$ or $\left\lfloor n/5\right\rfloor+1$. The edge sets $E(X)$ and $E(Y)$ are defined as follows. Suppose $x\in A_i$, $y\in A_j$, and $x\neq y$. We put $\{x,y\}\in E(X)$ if and only if $i-j\not\equiv\pm 2\pmod 5$, and we have $\{x,y\}\in E(Y)$ if and only if $i-j\not\equiv\pm 1\pmod 5$. It is straightforward to check that $\delta(X)$ and $\delta(Y)$ are each at least $3n/5-11/5$. Therefore, if we can show that $\FS(X,Y)$ is disconnected, then it will follow that $d_n\geq 3n/5-2$.

If $\sigma: [n]\to [n]$ is a bijection such that $\sigma(A_i)=A_i$ for all $i\in\{1,\ldots,5\}$ and $\sigma'$ is a bijection obtained by applying an $(X,Y)$-friendly swap to $\sigma$, then we must also have $\sigma'(A_i)=A_i$ for all $i\in\{1,\ldots,5\}$. Since $n\geq 5$, it follows that there exists a bijection $\tau:[n]\to [n]$ that is not in the same connected component of $\FS(X,Y)$ as the identity bijection $\Id:[n]\to [n]$.  
\end{proof}

\subsection{Upper bound}

The proof of the upper bound for $d_n$ given in Theorem~\ref{Thm:mindegree} is quite a bit more involved than the proof of the lower bound and will require the following lemma. 

\begin{lemma}\label{Lem:9/14}
Let $m\geq 1$, and let $G$ and $H$ be $m$-vertex graphs such that $\delta(G)$ and $\delta(H)$ are each at least $9m/14+1$. Let $\tau: V(G)\to V(H)$ be a bijection, and suppose $u,v\in V(H)$ are such that $\{\tau^{-1}(u),\tau^{-1}(v)\}\in E(G)$. Either $u$ and $v$ are $(G,H)$-exchangeable from $\tau$, 
or there exist vertices $w,x\in V(H)$ satisfying the following:
\begin{itemize}
\item $w\in N(u)\cap N(v)$.
\item $\tau^{-1}(x)\in N(\tau^{-1}(u))\cap N(\tau^{-1}(v))$.
\item $\{w,x\}\in E(H)$.
\item The induced subgraph $G\vert_{\tau^{-1}(N[w])}$ has vertex-disjoint connected subgraphs $\mathcal C_1$ and $\mathcal C_2$ such that $\tau^{-1}(u),\tau^{-1}(v),\tau^{-1}(x)\in V(\mathcal C_1)$ and $\tau^{-1}(w)\in V(\mathcal C_2)$. Furthermore, $2m/7\leq |V(\mathcal C_1)|\leq 3m/7$, $2m/7\leq |V(\mathcal C_2)|\leq 3m/7$, and each of $\mathcal C_1$ and $\mathcal C_2$ is Wilsonian. 
\end{itemize}
\end{lemma}

\begin{proof}
In order to illustrate where the number $9/14$ arises, we let $c=9/14$. To ease notation, let $u'=\tau^{-1}(u)$ and $v'=\tau^{-1}(v)$. Let $A'=N(u')\cap N(v')$, $B=N(u)\cap N(v)$, $C=N[B]$, and $D'=N(u')\cup N(v')$. Let $A=\tau(A')$, $B'=\tau^{-1}(B)$, $C'=\tau^{-1}(C)$, and $D=\tau(D')$. Note that 
\begin{equation}\label{Eq1}
|A|=|A'|\geq (2c-1)m,
\end{equation}
and
\begin{equation}\label{Eq2}
|B|=|B'|\geq (2c-1)m.
\end{equation}
We have $|D'|=|N(u')|+|N(v')|-|N(u') \cap N(v')|=|N(u')|+|N(v')|-|A'|$, so 
\begin{equation}\label{Eq3}
|D|=|D'|\geq 2cm-|A'|.
\end{equation}
There are two cases to consider. 

\noindent {\bf Case 1:} $A'\cap C'=\emptyset$. In this case, we have the inequality
\begin{equation}\label{Eq4}
|A'|+|C'|\leq m.
\end{equation} 
We may assume without loss of generality that \[|B' \cap N(u')|\geq |B'\cap N(v')|.\] Thus, letting $Z=\tau(B'\cap N(u'))=B\cap \tau(N(u'))$, we have \[|Z|=|B'\cap N(u')|\geq \frac{1}{2}(|B' \cap N(u')|+|B'\cap N(v')|)\geq\frac{1}{2}|B'\cap D'|=\frac{1}{2}|B'\setminus(V(G)\setminus D')|\] 
\begin{equation}\label{Eq5}
\geq \frac{1}{2}(|B'|-|V(G)\setminus D'|)=\frac{1}{2}(|B|+|D'|-m).
\end{equation} We are going to show that the induced subgraph $H\vert_Z$ is Wilsonian. Lemma~\ref{lem:wilsonian} tells us that it suffices to show that $|N(z)\cap Z|\geq \dfrac{1}{2}|Z|$ for all $z\in Z$. Choose $z\in Z$. Since $Z\subseteq B$, we have $N(z)\subseteq C$, so \[|N(z)\cap Z|=|N(z)|-|N(z)\setminus Z|\geq cm-|N(z)\setminus Z|\geq cm-(|C|-|Z|).\] We want to show that $cm-(|C|-|Z|)\geq \dfrac{1}{2}|Z|$, which we can rewrite as $cm-|C|+\dfrac{1}{2}|Z|\geq 0$. Now, using \eqref{Eq1}, \eqref{Eq2}, \eqref{Eq3}, \eqref{Eq4}, and \eqref{Eq5}, we find that \[cm-|C|+\dfrac{1}{2}|Z|\geq cm-|C|+\frac{1}{4}(|B|+|D'|-m)\geq cm-|C|+\frac{1}{4}(|B|+(2cm-|A'|)-m)\] \[=cm-|C|+\frac{1}{4}|B|+\frac{1}{4}(2cm-m)-\frac{1}{4}|A'|=cm-(|C'|+|A'|)+\frac{1}{4}|B|+\frac{1}{4}(2cm-m)+\frac{3}{4}|A'|\] \[\geq cm-m+\frac{1}{4}|B|+\frac{1}{4}(2cm-m)+\frac{3}{4}|A'|\geq cm-m+\frac{1}{4}(2c-1)m+\frac{1}{4}(2cm-m)+\frac{3}{4}(2c-1)m\] \[=\frac{7}{2}cm-\frac{9}{4}m=0.\] (This is where the number $9/14$ arises.) This demonstrates that $H\vert_Z$ is Wilsonian. Since $Z\subseteq B=N(u)\cap N(v)$, it follows from Definition~\ref{def:wilsonian} that the induced subgraph $H\vert_{Z\cup\{u,v\}}$ is also Wilsonian. Furthermore, the induced subgraph $G\vert_{\tau^{-1}(Z\cup\{u,v\})}=G\vert_{(B'\cap N(u'))\cup\{u',v'\}}$ contains a spanning star with center $u'$. By Wilson's Theorem~\ref{thm:wilson}, the graph $\FS(G\vert_{\tau^{-1}(Z\cup\{u,v\})},H\vert_{Z\cup\{u,v\}})$ is connected. Therefore, there is a sequence $\Sigma$ of $(G\vert_{\tau^{-1}(Z\cup\{u,v\})},H\vert_{Z\cup\{u,v\}})$-friendly swaps that exchanges $u$ and $v$ when applied to $\tau\vert_{\tau^{-1}(Z\cup\{u,v\})}$. We can view $\Sigma$ as a sequence of $(G,H)$-friendly swaps that exchanges $u$ and $v$ when applied to $\tau$; consequently, $u$ and $v$ are $(G,H)$-exchangeable from $\tau$.  

\noindent {\bf Case 2:} $A'\cap C'\neq\emptyset$. In this case, there exist $w\in B$ and $x\in N[w]$ such that $\tau^{-1}(x)\in A'=N(u')\cap N(v')$. If $x=w$, then applying the sequence of $(G,H)$-friendly swaps $wu,wv,wu$ to $\tau$ exchanges $u$ and $v$, implying that $u$ and $v$ are $(G,H)$-exchangeable from $\tau$. Now suppose $x\neq w$ so that $x\in N(w)$.  

First, assume that there exists a path $\tau^{-1}(w),\tau^{-1}(t_1),\tau^{-1}(t_2),\ldots,\tau^{-1}(t_r),\tau^{-1}(x)$ in $G$ such that $t_1,\ldots,t_r\in N(w)\setminus\{u,v\}$.  Applying the sequence of $(G,H)$-friendly swaps $$wt_1, wt_2,\ldots, wt_r, wx, wu, wv, wu, wx, wt_r,\ldots,wt_2,wt_1$$ to $\tau$ exchanges $u$ and $v$, implying that $u$ and $v$ are $(G,H)$-exchangeable from $\tau$ in this case.

Second, assume that there exist paths $\tau^{-1}(w),\tau^{-1}(t_1),\tau^{-1}(t_2),\ldots,\tau^{-1}(t_r),\tau^{-1}(u),\tau^{-1}(x)$ and $\tau^{-1}(w),\tau^{-1}(\widetilde t_1),\tau^{-1}(\widetilde t_2),\ldots,\tau^{-1}(\widetilde t_r),\tau^{-1}(v),\tau^{-1}(x)$ in $G$ such that $t_1,\ldots,t_r,\widetilde t_1,\ldots,\widetilde t_r\in N(w)\setminus\{u,v\}$. Applying the sequence of $(G,H)$-friendly swaps \[wt_1, wt_2,\ldots, wt_r, wu, wx, wv, wx, wu, wt_r,\ldots,wt_2,wt_1,\]
\[w\widetilde t_1, w\widetilde t_2,\ldots, w\widetilde t_r, wx, wv, wu, wv, wx, w\widetilde t_r,\ldots,w\widetilde t_2,w\widetilde t_1,\] \[wt_1, wt_2,\ldots, wt_r, wv, wu, wx, wu, wv, wt_r,\ldots,wt_2,wt_1\] to $\tau$ exchanges $u$ and $v$, implying that $u$ and $v$ are $(G,H)$-exchangeable from $\tau$ in this case as well.

The last remaining subcase is that in which $w$ and $x$ are in different connected components of either $G\vert_{\tau^{-1}(N[w]\setminus\{u\})}$ or $G\vert_{\tau^{-1}(N[w]\setminus\{v\})}$; without loss of generality, we may assume they are in different connected components of $G\vert_{\tau^{-1}(N[w]\setminus\{u\})}$. Let $\mathcal C_2$ be the connected component of $G\vert_{\tau^{-1}(N[w]\setminus\{u\})}$ containing $\tau^{-1}(w)$, and let $\widetilde{\mathcal C}_1$ be the connected component of $G\vert_{\tau^{-1}(N[w]\setminus\{u\})}$ containing $\tau^{-1}(x)$. Note that $\tau^{-1}(v)\in V(\widetilde{\mathcal C}_1)$. 

Each vertex in $\widetilde{\mathcal C}_1$ has at least $cm+1$ neighbors in $G$, and at most $m-|N[w]|$ of these neighbors lie outside $\tau^{-1}(N[w])$. Each neighbor of a vertex in $\widetilde{\mathcal C}_1$ that is in $\tau^{-1}(N[w])$ must be either equal to $\tau^{-1}(u)$ or in $\widetilde{\mathcal C}_1$. Therefore, the graph $\widetilde{\mathcal C}_1$ has minimum degree at least $cm+1-(m-|N[w]|)-1=(c-1)m+|N[w]|$.  It follows that \[|V(\widetilde{\mathcal C}_1)|\geq(c-1)m+|N[w]|\geq (2c-1)m=\dfrac{2}{7}m.\] The same argument shows that  $\mathcal C_2$ has minimum degree at least $(c-1)m+|N[w]|$ and that $|V(\mathcal C_2)|\geq 2m/7$.  Now, $|N[w]|-1\geq |V(\widetilde{\mathcal C}_1)|+|V(\mathcal C_2)|\geq 2((c-1)m+|N[w]|)$, so $|N[w]|\leq 2(1-c)m-1=5m/7-1$. Because $|V(\widetilde{\mathcal C}_1)|\geq 2m/7$, this shows that $|V(\mathcal C_2)|\leq |N[w]|-|V(\widetilde{\mathcal C}_1)|\leq 3m/7-1$. The minimum degree of $\mathcal C_2$ satisfies \[\delta(\mathcal C_2)\geq (c-1)m+|N[w]|\geq (2c-1)m=\dfrac{2}{7}m>\dfrac{1}{2}|V(\mathcal C_2)|,\] so Lemma~\ref{lem:wilsonian} tells us that $\mathcal C_2$ is Wilsonian. A similar argument shows that $\widetilde{\mathcal C}_1$ is Wilsonian and satisfies $|V(\mathcal C_1)|\leq 3m/7-1$. Finally, let $\mathcal C_1$ be the induced subgraph of $G$ on the vertex set $V(\widetilde{\mathcal C}_1)\cup\{\tau^{-1}(u)\}$. We have $2m/7\leq |V(\mathcal C_1)|\leq 3m/7$. Since $\tau^{-1}(u)$ is adjacent to $\tau^{-1}(v)$ and $\tau^{-1}(x)$, both of which are vertices in the Wilsonian graph $\widetilde{\mathcal C}_1$, it follows that $\mathcal C_1$ is Wilsonian. 
\end{proof}

\begin{proposition}
If $n\geq 16$, then $d_n\leq \dfrac{9}{14}n+2$. 
\end{proposition}

\begin{proof}
Let $X$ and $Y$ be $n$-vertex graphs such that $\delta(X)\geq 9n/14+2$ and $\delta(Y)\geq 9n/14+2$. We will consider subgraphs of $X$ and $Y$, but the notation $N(\cdot)$ and $N[\cdot]$ will always refer to open and closed neighborhoods in the full graphs $X$ and $Y$. Fix a bijection $\sigma: V(X)\to V(Y)$ and vertices $u,v\in V(Y)$ such that $\{\sigma^{-1}(u),\sigma^{-1}(v)\}\in E(X)$. We will prove that $u$ and $v$ are $(X,Y)$-exchangeable from $\sigma$. As $X$ is certainly connected and $u$ and $v$ were chosen arbitrarily, it will then follow from Lemma~\ref{lem:exchangeable-K_n} that $\FS(X,Y)$ is connected. This will prove that $d_n\leq 9n/14+2$. 

Suppose instead that $u$ and $v$ are not $(X,Y)$-exchangeable from $\sigma$. Applying Lemma~\ref{Lem:9/14} (with $X$, $Y$, and $\sigma$ playing the roles of $G$, $H$, and $\tau$), we find that there exist $w_0,x_0\in V(Y)$ such that the following hold:
\begin{itemize}
\item $w_0\in N(u)\cap N(v)$.
\item $\sigma^{-1}(x_0)\in N(\sigma^{-1}(u))\cap N(\sigma^{-1}(v))$.
\item $\{w_0,x_0\}\in E(Y)$.
\item The induced subgraph $X\vert_{\sigma^{-1}(N[w_0])}$ has vertex-disjoint connected subgraphs $\mathcal C_1$ and $\mathcal C_2$ such that $\sigma^{-1}(u),\sigma^{-1}(v),\sigma^{-1}(x_0)\in V(\mathcal C_1)$ and $\tau^{-1}(w_0)\in V(\mathcal C_2)$. Furthermore, $2m/7\leq |V(\mathcal C_1)|\leq 3m/7$, $2m/7\leq |V(\mathcal C_2)|\leq 3m/7$, and each of $\mathcal C_1$ and $\mathcal C_2$ is Wilsonian.
\end{itemize}

To ease notation, let us write $u'=\sigma^{-1}(u)$, $v'=\sigma^{-1}(v)$, $w_0'=\sigma^{-1}(w_0)$, and $x_0'=\sigma^{-1}(x_0)$. Let us also write $\sigma(\mathcal C_i)$ for the induced subgraph of $Y$ on the vertex set $\sigma(V(\mathcal C_i))$. 
Let $m=n-2$, and let $H=X\vert_{V(X)\setminus\{u',v'\}}$ and $G=Y\vert_{V(Y)\setminus\{u,v\}}$. Let $\tau:V(G)\to V(H)$ be the restriction of $\sigma^{-1}$ to the set $V(G)=V(Y)\setminus\{u,v\}$. Suppose $w_0$ and $x_0$ are $(H,G)$-exchangeable from $\tau^{-1}$. This means that there is a sequence $\Sigma$ of $(H,G)$-friendly swaps such that applying $\Sigma$ to $\tau^{-1}$ exchanges $w_0$ and $x_0$. We may view $\Sigma$ as a sequence of $(X,Y)$-friendly swaps that does not involve $u$ or $v$. Applying the sequence of $(X,Y)$-friendly swaps $\Sigma, w_0u, w_0v, w_0u,\rev(\Sigma)$ to $\sigma$ exchanges $u$ and $v$. This contradicts the assumption that $u$ and $v$ are not $(X,Y)$-exchangeable from $\sigma$, so we conclude that the vertices $w_0$ and $x_0$ are not $(H,G)$-exchangeable from $\tau^{-1}$. By Lemma~\ref{lem:exchangeabilityinverse}, the vertices $w_0'$ and $x_0'$ are not $(G,H)$-exchangeable from $\tau$. Note that $G$ and $H$ are $m$-vertex graphs such that $\delta(G)$ and $\delta(H)$ are at least $(9n/14+2)-2\geq 9m/14+1$. Furthermore, $w_0',x_0'\in V(H)$ are such that $\{\tau^{-1}(w_0'),\tau^{-1}(x_0')\}=\{w_0,x_0\}\in E(G)$. It now follows from Lemma~\ref{Lem:9/14} that there exist vertices $w_1',x_1'\in V(H)=V(X)\setminus\{u',v'\}$ such that, using the notation $w_1=\sigma(w_1')$ and $x_1=\sigma(x_1')$, we have:
\begin{itemize}
\item $w_1'\in N(w_0')\cap N(x_0')$.
\item $\tau^{-1}(x_1')\in N(\tau^{-1}(w_0'))\cap N(\tau^{-1}(x_0'))=N(w_0)\cap N(x_0)$.
\item $\{w_1',x_1'\}\in E(H)\subseteq E(X)$.
\item The induced subgraph $G\vert_{\tau^{-1}(N[w_1'])}=Y\vert_{\sigma(N[w_1'])}$ has vertex-disjoint connected subgraphs $\mathcal D_1$ and $\mathcal D_2$ such that the vertices $\tau^{-1}(w_0')=w_0$, $\tau^{-1}(x_0')=x_0$, and $\tau^{-1}(x_1')=x_1$ are in $\mathcal D_1$ and the vertex $\tau^{-1}(w_1')=w_1$ is in $\mathcal D_2$. Furthermore, $2m/7\leq |V(\mathcal D_1)|\leq 3m/7$, $2m/7\leq |V(\mathcal D_2)|\leq 3m/7$, and each of $\mathcal D_1$ and $\mathcal D_2$ is Wilsonian.
\end{itemize}
Let us write $\sigma^{-1}(\mathcal D_i)$ for the induced subgraph of $X$ on the vertex set $\sigma^{-1}(V(\mathcal D_i))$. Because $w_1'\in \sigma^{-1}(V(\mathcal D_2))\subseteq N[w_1']$, the graph $\sigma^{-1}(\mathcal D_2)$ contains a spanning star with center $w_1'$. 

Figure~\ref{Fig1} shows a schematic diagram indicating some of the special vertices and sets of vertices in $X$ and $Y$ that we have considered up to this point. Note that the circle representing $\sigma^{-1}(\mathcal D_2)$ is drawn so that the vertex set of this subgraph appears to be disjoint from those of $\mathcal C_1$ and $\mathcal C_2$. Our next goal is to prove that $V(\mathcal D_2)\cap N[w_0]=\emptyset$ (so that the figure is indeed accurate). 

\begin{figure}[ht]
\begin{center}
\includegraphics[height=6cm]{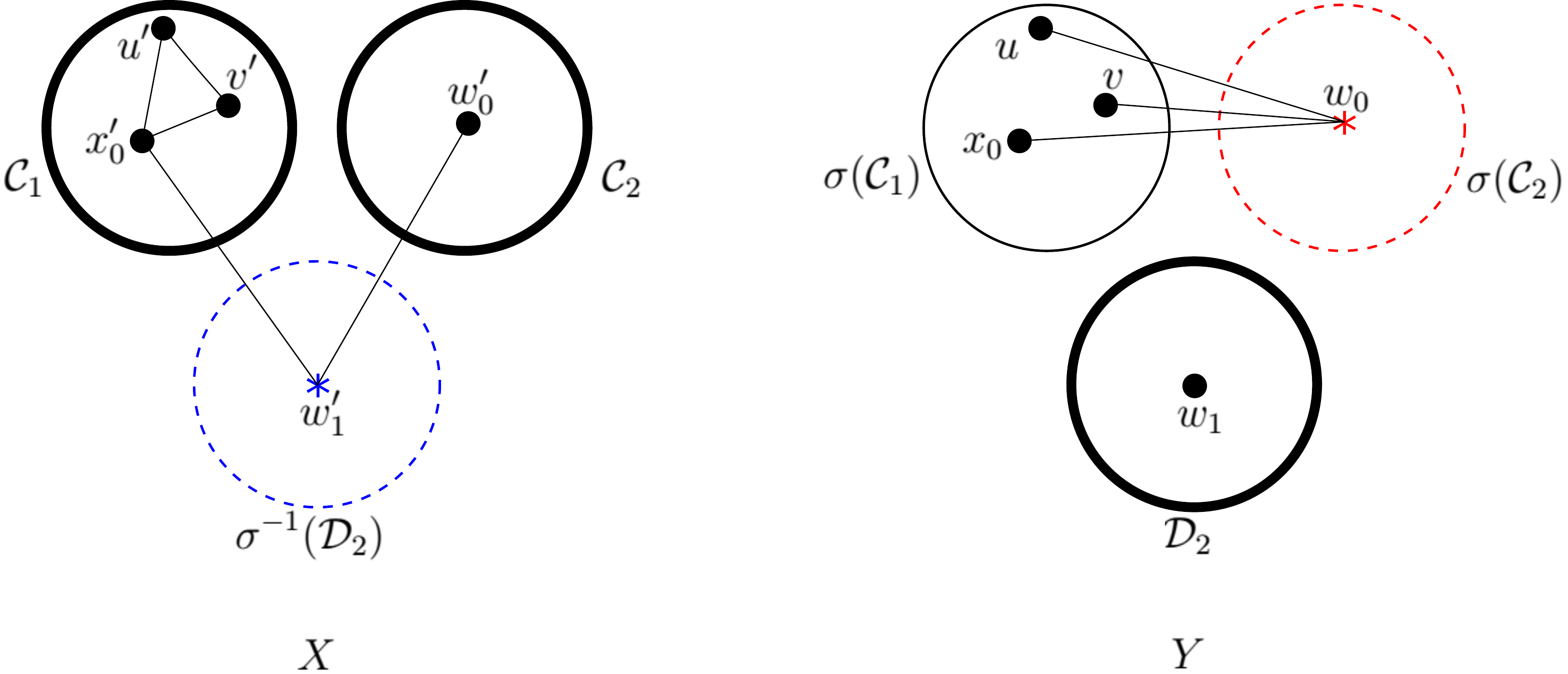}
\caption{In each of $X$ and $Y$, the three circles represent induced subgraphs. A thick circle indicates a Wilsonian induced subgraph. A colored dotted circle indicates an induced subgraph with a spanning star, and the asterisk of the same color marks the center of that star.}
\label{Fig1}
\end{center}  
\end{figure}

Suppose there exists $q\in V(\mathcal D_2)\cap N[w_0]$. Because $\mathcal D_2$ is Wilsonian and $\sigma^{-1}(\mathcal D_2)$ contains a spanning star with center $w_1'$, we can use Wilson's Theorem~\ref{thm:wilson} to deduce that there is a sequence $\Sigma$ of $(\sigma^{-1}(\mathcal D_2),\mathcal D_2)$-friendly swaps that transforms $\sigma\vert_{\sigma^{-1}(V(\mathcal D_2))}$ into a bijection that sends $w_1'$ to $q$. If $q=w_0$, then let $\Sigma'=\Sigma$. If $q\neq w_0$, then let $\Sigma'$ be the sequence $\Sigma,qw_0$. In either case, we can view $\Sigma'$ as a sequence of $(X,Y)$-friendly swaps that does not involve $u$ or $v$ and that transforms $\sigma$ into a bijection $\mu: V(X)\to V(Y)$ satisfying $\mu(w_1')=w_0$. Applying the sequence of $(X,Y)$-friendly swaps \[\Sigma',w_0x_0,w_0u,w_0v,w_0u,w_0x_0,\rev(\Sigma')\] to $\sigma$ exchanges $u$ and $v$, which contradicts the assumption that $u$ and $v$ are not $(X,Y)$-exchangeable from $\sigma$. Therefore, we must have \[V(\mathcal D_2)\cap N[w_0]=\emptyset.\] 

Consider a vertex $h\in V(\mathcal D_2)$. Suppose by way of contradiction that there exist distinct vertices $s_1,s_2\in N(h)\cap\sigma(V(\mathcal C_2))$. Because $\delta(X)$ and $\delta(Y)$ are at least $9n/14+2$, we can certainly find distinct vertices $t_1,t_2\in N(w_0)\cap\sigma(N(w_0'))$. Note that the vertices $t_1'=\sigma^{-1}(t_1)$ and $t_2'=\sigma^{-1}(t_2)$ are in the same connected component of $X\vert_{\sigma^{-1}(N[w_0])}$ as $w_0'$, namely, $\mathcal C_2$. Because $\mathcal C_2$ is Wilsonian and $\sigma(\mathcal C_2)$ contains a spanning star with center $w_0$, we can invoke Wilson's Theorem~\ref{thm:wilson} in order to see that there is a sequence $\Sigma$ of $(\mathcal C_2,\sigma(\mathcal C_2))$-friendly swaps that transforms $\sigma\vert_{V(\mathcal C_2)}$ into a bijection that sends $t_1'$ to $s_1$, sends $t_2'$ to $w_0$, and sends $w_0'$ to $s_2$. We also know that $\mathcal D_2$ is Wilsonian and that $\sigma^{-1}(\mathcal D_2)$ contains a spanning star with center $w_1'$, so we can use Wilson's Theorem~\ref{thm:wilson} once again to obtain a sequence $\Sigma'$ of $(\sigma^{-1}(\mathcal D_2),\mathcal D_2)$-friendly swaps that transforms $\sigma\vert_{\sigma^{-1}(V(\mathcal D_2))}$ into a bijection that sends $w_1'$ to $h$. We have seen that $V(\mathcal D_2)\cap N[w_0]=\emptyset$, so the sequences $\Sigma$ and $\Sigma'$ involve disjoint sets of vertices. Furthermore, $\Sigma$ and $\Sigma'$ do involve $u$, $v$, or $x_0$.  Therefore, we may view $\Sigma,\Sigma'$ as a sequence of $(X,Y)$-friendly swaps that transforms $\sigma$ into a bijection $\mu$ satisfying \[\mu(u')=u, \quad \mu(v')=v,\quad\mu(x_0')=x_0,\quad\mu(t_1')=s_1,\quad\mu(t_2')=w_0,\quad\mu(w_0')=s_2,\quad\mu(w_1')=h.\] However, we now readily check that applying the sequence \[\Sigma,\Sigma',hs_2,hs_1,w_0s_1,w_0s_2,
w_0x_0,w_0u,w_0v,w_0u,w_0x_0,w_0s_2,w_0s_1,hs_1,hs_2,\rev(\Sigma'),\rev(\Sigma)\] of $(X,Y)$-friendly swaps to $\sigma$ exchanges $u$ and $v$, which is a contradiction. From this contradiction, we deduce that $|N(h)\cap\sigma(V(\mathcal C_2))|\leq 1$. As a consequence, we find that $|N(h)\cap (V(Y)\setminus \sigma(V(\mathcal C_2)))|\geq 9n/14+1$. As $|V(Y)\setminus(V(\mathcal C_2))|=n-|V(\mathcal C_2)|\leq 5n/7$, there are at most $n/14-1$ vertices in $V(Y)\setminus \sigma(V(\mathcal C_2))$ that are not adjacent to $h$. Because $|V(\mathcal C_1)\setminus\{u,v\}|\geq 2n/7-2$, it follows that \[|N(h)\cap(V(\mathcal C_1)\setminus\{u,v\})|\geq |V(\mathcal C_1)\setminus\{u,v\}|-\left(\frac{1}{14}n-1\right)\geq \frac{3}{4}|V(\mathcal C_1)\setminus\{u,v\}|+\frac{1}{2}.\] As $h$ was arbitrary, this demonstrates that every vertex in $\mathcal D_2$ has at least $\dfrac{3}{4}|V(\mathcal C_1)\setminus\{u,v\}|+\dfrac{1}{2}$ neighbors in $V(\mathcal C_1)\setminus\{u,v\}$. 

Let $Q=V(X)\setminus(V(\mathcal C_1)\cup V(\mathcal C_2)\cup\sigma^{-1}(V(\mathcal D_2)))$. We have seen that $V(\mathcal C_1)$, $V(\mathcal C_2)$, and $\sigma^{-1}(V(\mathcal D_2))$ are disjoint (since $V(\mathcal D_2)\cap N[w_0]=\emptyset$), so \[|Q|=n-|V(\mathcal C_1)|-|V(\mathcal C_2)|-|V(\mathcal D_2)|\leq n-|V(\mathcal C_1)|-\frac{2}{7}n-\frac{2}{7}m=\frac{3}{7}n-|V(\mathcal C_1)|+\frac{4}{7}.\] This shows that if $t\in V(\mathcal C_1)$, then the number of neighbors of $t$ in $V(\mathcal C_2)\cup\sigma^{-1}(V(\mathcal D_2))$ satisfies \[|N(t)\cap(V(\mathcal C_2)\cup\sigma^{-1}(V(\mathcal D_2)))|\geq \frac{9}{14}n+2-|N(t)\cap V(\mathcal C_1)|-|N(t)\cap Q|\] \[\geq \frac{9}{14}n+2-(|V(\mathcal C_1)|-1)-\left(\frac{3}{7}n-|V(\mathcal C_1)|+\frac{4}{7}\right)=\frac{3}{14}n+\frac{17}{7}.\] Since $|\sigma^{-1}(\mathcal D_2)|\leq 3m/7=3(n-2)/7$ and each vertex in $\mathcal C_1$ has at least $3n/14+17/7$ neighbors in $\sigma^{-1}(\mathcal D_2)$, we can use the pigeonhole principle to see that there exists a vertex $y'\in \sigma^{-1}(V(\mathcal D_2))$ with at least \[\frac{|V(\mathcal C_1)|(3n/14+17/7)}{3(n-2)/7}\geq\frac{1}{2}|V(\mathcal C_1)|+\frac{17}{3(n-2)}|V(\mathcal C_1)|\geq\frac{1}{2}|V(\mathcal C_1)|+\frac{17}{3(n-2)}\cdot\frac{2}{7}n\] \[\geq \frac{1}{2}|V(\mathcal C_1)\setminus\{u,v\}|+\frac{55}{21}\] neighbors in $\mathcal C_1$. Letting $Z'=N(y')\cap(V(\mathcal C_1)\setminus\{u,v\})$, we have \[|Z'|\geq \frac{1}{2}|V(\mathcal C_1)\setminus\{u,v\}|+\frac{13}{21}.\]

We have seen that each vertex in $\mathcal D_2$ has at least $\dfrac{3}{4}|V(\mathcal C_1)\setminus\{u,v\}|+\dfrac{1}{2}$ neighbors in $V(\mathcal C_1)\setminus\{u,v\}$. Among these neighbors, at most $\dfrac{1}{2}|V(\mathcal C_1)\setminus\{u,v\}|-\dfrac{13}{21}$ are not in the set $Z=\sigma(Z')$. Therefore, each vertex in $\mathcal D_2$ has at least \[\dfrac{3}{4}|V(\mathcal C_1)\setminus\{u,v\}|+\dfrac{1}{2}-\left(\frac{1}{2}|V(\mathcal C_1)\setminus\{u,v\}|-\frac{13}{21}\right)=\frac{1}{4}|V(\mathcal C_1)\setminus\{u,v\}|+\frac{47}{42}\]
neighbors in $Z$. Because $n\geq 16$, we have $|\mathcal D_2|\geq 2m/7=2(n-2)/7\geq 4$, so there must exist distinct vertices $h_1$ and $h_2$ in $\mathcal D_2$ that have a common neighbor $z$ in $Z$. We once again use the fact that $\mathcal D_2$ is Wilsonian and that $\sigma^{-1}(\mathcal D_2)$ contains a spanning star with center $w_1'$. By Wilson's Theorem~\ref{thm:wilson}, there is a sequence $\Sigma$ of $(\sigma^{-1}(\mathcal D_2),\mathcal D_2)$-friendly swaps that transforms $\sigma\vert_{\sigma^{-1}(\mathcal D_2)}$ into a bijection that sends $y'$ to $h_1$ and sends $w_1'$ to $h_2$. Notice that $\Sigma$ does not involve any of the vertices $u,v,w_0,x_0,z$ since these vertices are not in $\mathcal D_2$. Therefore, we can view $\Sigma$ as a sequence of $(X,Y)$-friendly swaps that transforms $\sigma$ into a bijection $\mu$ satisfying \[\mu(u')=u, \quad \mu(v')=v,\quad\mu(x_0')=x_0,\quad\mu(w_0')=w_0,\quad\mu(y')=h_1,\quad\mu(w_1')=h_2.\]  We now check that applying the sequence \[\Sigma,zh_1,zh_2,zw_0,w_0x_0,w_0u,w_0v,w_0u,w_0x_0,zw_0,zh_2,zh_1,\rev(\Sigma)\] of $(X,Y)$-friendly swaps to $\sigma$ exchanges $u$ and $v$, which is our final contradiction. 
\end{proof}

\section{Bipartite Graphs with Large Minimum Degree}\label{Sec:MinDegreeBipartite}
For $r\geq2$, recall that $d_{r,r}$ is the smallest nonnegative integer such that any two edge-subgraphs $X$ and $Y$ of $K_{r,r}$ with minimum degrees at least $d_{r,r}$ must have a friends-and-strangers graph $\FS(X,Y)$ with exactly $2$ connected components. In this section, we prove Theorem~\ref{Thm:bipartitemindegree}, which almost determines the exact value of $d_{r,r}$.  

\subsection{Lower bound}

\begin{proposition}
For $r\geq 2$, we have $d_{r,r}\geq\left\lceil\dfrac{3r+1}{4}\right\rceil$. 
\end{proposition}
\begin{proof}
For each $r \geq 2$, we construct edge-subgraphs $X$ and $Y$ of $K_{r,r}$, each with minimum degree at least $\left\lceil\frac{3r+1}{4}\right\rceil-1$, such that $\FS(X,Y)$ has an isolated vertex; by Proposition~\ref{prop:bipartite}, this suffices to show that $\FS(X,Y)$ has more than $2$ connected components.  We will partition each partite set in each graph into two subsets.  Let the partite sets in the bipartition of $X$ be $A_X \cupdot B_X$ and $C_X \cupdot D_X$, where $A_X$ and $D_X$ each contain $\lceil r/2 \rceil$ vertices and $B_X$ and $C_X$ each contain $\lfloor r/2 \rfloor$ vertices; similarly, let the partite sets in the bipartition of $Y$ be $A_Y \cupdot C_Y$ and $B_Y \cupdot D_Y$, where $A_Y$ and $D_Y$ each contain $\lceil r/2 \rceil$ vertices and $B_Y$ and $C_Y$ each contain $\lfloor r/2 \rfloor$ vertices.  Fix a bijection $\sigma_0: V(X) \to V(Y)$ such that $\sigma_0(A_X)=A_Y$, $\sigma_0(B_X)=B_Y$, $\sigma_0(C_X)=C_Y$, and $\sigma_0(D_X)=D_Y$.

We now describe the edge set of $X$.  First, connect every vertex in $A_X$ to every vertex in $C_X$, and connect every vertex in $B_X$ to every vertex in $D_X$.  Second, add edges between $A_X$ and $D_X$ so that each vertex in $A_X$ has $\left\lfloor \frac{\lceil r/2 \rceil}{2} \right\rfloor$ neighbors in $D_X$ and each vertex in $D_X$ has $\left\lfloor \frac{\lceil r/2 \rceil}{2} \right\rfloor$ neighbors in $A_X$ (this is easy to achieve); similarly, add edges between $B_X$ and $C_X$ so that each vertex in $B_X$ has $\left\lfloor \frac{\lfloor r/2 \rfloor}{2} \right\rfloor$ neighbors in $C_X$ and each vertex in $C_X$ has $\left\lfloor \frac{\lfloor r/2 \rfloor}{2} \right\rfloor$ neighbors in $B_X$.  We now describe the edge set of $Y$ in a similar manner.  First, connect every vertex in $A_Y$ to every vertex in $B_Y$, and connect every vertex in $C_Y$ to every vertex in $D_Y$.  Second, connect $u \in A_Y$ to $v \in D_Y$ if and only if $\{\sigma_0^{-1}(u),\sigma_0^{-1}(v)\}$ is not an edge of $X$; similarly, connect $s \in B_Y$ to $t \in C_Y$ if and only if $\{\sigma_0^{-1}(s),\sigma_0^{-1}(t)\}$ is not an edge of $X$.  By construction, $\sigma_0$ is an isolated vertex in $\FS(X,Y)$.  It is straightforward to check (with casework on the residue of $r$ modulo $4$) that $X$ and $Y$ each have minimum degree at least $\left\lceil\frac{3r+1}{4}\right\rceil-1$.
\end{proof}

\subsection{Upper bound}

\begin{proposition}\label{propmindegreebipartite}
Let $X$ and $Y$ be edge-subgraphs of $K_{r,r}$ such that $\min\{\delta(X),\delta(Y)\}\geq\left\lceil\frac{3r+2}{4}\right\rceil$. Let $\sigma: V(X)\to V(Y)$ be a bijection. If $u,v\in V(Y)$ are in different partite sets of $Y$ and are such that $\{\sigma^{-1}(u),\sigma^{-1}(v)\}\in E(X)$, then $u$ and $v$ are $(X,Y)$-exchangeable from $\sigma$. 
\end{proposition}

Before proving this proposition, let us see how it implies the desired inequality $d_{r,r}\leq\left\lceil\frac{3r+2}{4}\right\rceil$. Suppose $X$ and $Y$ are edge-subgraphs of $K_{r,r}$ with $\min\{\delta(X),\delta(Y)\}\geq\left\lceil\frac{3r+2}{4}\right\rceil$. Proposition~\ref{propmindegreebipartite} tells us that the hypothesis of Proposition~\ref{lem:exchangeable} is satisfied with $\widetilde Y=K_{r,r}$, so it follows from that lemma that the number of connected components of $\FS(X,K_{r,r})$ is the same as the number of connected components of $\FS(X,Y)$. We also know that $\FS(X,K_{r,r})\cong\FS(K_{r,r},X)$. Because $K_{r,r}$ and $X$ are edge-subgraphs of $K_{r,r}$ with $\min\{\delta(K_{r,r}),\delta(X)\}\geq\left\lceil\frac{3r+2}{4}\right\rceil$, we can use Proposition~\ref{propmindegreebipartite} once again to see that the hypothesis of Proposition~\ref{lem:exchangeable} is satisfied with the pair $(K_{r,r},X)$ playing the role of $(X,Y)$ and with $\widetilde Y=K_{r,r}$. Therefore, the number of connected components of $\FS(K_{r,r},X)$, which is also the number of connected components of $\FS(X,Y)$, is equal to the number of connected components of $\FS(K_{r,r},K_{r,r})$. We know by Proposition~\ref{prop:complete-bipartite} that $\FS(K_{r,r},K_{r,r})$ has $2$ connected components. As $X$ and $Y$ were arbitrary, this proves that $d_{r,r}\leq\left\lceil\frac{3r+2}{4}\right\rceil$. 

\begin{proof}[Proof of Proposition~\ref{propmindegreebipartite}]
Let $\delta=\min\{\delta(X),\delta(Y)\}$; our hypothesis states that $\delta\geq\left\lceil\frac{3r+2}{4}\right\rceil$. Let $\{A_X,B_X\}$ and $\{A_Y, B_Y\}$ be the bipartitions of $X$ and $Y$, respectively. Without loss of generality, we may assume $u\in A_Y$ and $v\in B_Y$. Let $u'=\sigma^{-1}(u)$ and $v'=\sigma^{-1}(v)$. Our goal is to show that $\sigma$ and $\sigma\circ(u'\,\,v')$ are in the same connected component of $\FS(X,Y)$. We may assume that the partite set of $X$ containing $v'$ contains at least $r/2$ elements of $\sigma^{-1}(B_Y)$ since, otherwise, we can simply switch the roles of $\sigma$ and $\sigma\circ(u'\,\,v')$. Without loss of generality, we may assume $u'\in A_X$ and $v'\in B_X$. Thus, $|B_X\cap\sigma^{-1}(B_Y)|\geq r/2$.

Note that there are at most $r-\delta$ vertices $B_Y$ that are not adjacent to $u$, at most $r/2$ vertices in $B_Y$ that are not in $\sigma(B_X)$, and at most $r-\delta$ vertices in $B_X$ that are not adjacent to $u'$. Since $(r-\delta)+(r/2)+(r-\delta)=5r/2-2\delta<r$, there must exist $w\in N(u)\subseteq B_Y$ such that the vertex $w'=\sigma^{-1}(w)$ is adjacent to $u'$. Let $D=A_X\setminus(N(v')\cap N(w'))$. Note that $|D|\leq r-(2\delta-r)=2r-2\delta$.

Suppose we now consider an arbitrary bijection $\tau: V(X) \to V(Y)$. We have the following claims:

\noindent{\bf Claim 1:} If $x\in\tau(D)\cap A_Y$ and $|\tau^{-1}(A_Y)\cap A_X| > 2r-2\delta+1$, then there exists a vertex $y\in (B_Y\cap\tau(B_X))\setminus\{v,w\}$ such that $y$ is adjacent to $x$ and $\tau^{-1}(y)$ is adjacent to $\tau^{-1}(x)$. 

\noindent{\bf Claim 2:} If $q\in B_Y\cap \tau(A_X \setminus D)$ and $|\tau^{-1}(B_Y)\cap A_X| > 2r-2\delta$, then there exists $s\in A_Y\cap\tau(B_X)$ such that $s$ is adjacent to $q$ and $\tau^{-1}(s)$ is adjacent to $\tau^{-1}(q)$. 

To prove Claim 1, notice that there are \emph{strictly fewer} than $r-(2r-2\delta+1)=2\delta-r-1$ vertices in $A_Y\cap\tau(B_X)$, at most $r-\delta$ vertices in $B_X$ that are not adjacent to $\tau^{-1}(x)$, and at most $r-\delta$ vertices in $B_Y$ that are not adjacent to $x$. Since $(2\delta-r-1)+(r-\delta)+(r-\delta)=r-1$, it follows that there exist distinct vertices $y_1,y_2\in B_Y\cap\tau(B_X)$ such that $y_1,y_2\in N(x)$ and $\tau^{-1}(y_1),\tau^{-1}(y_2)\in N(\tau^{-1}(x))$. Because $\tau^{-1}(x)$ is in $D$ and is a common neighbor of $\tau^{-1}(y_1)$ and $\tau^{-1}(y_2)$, it follows from the definition of $D$ that at least one of the vertices $\tau^{-1}(y_1),\tau^{-1}(y_2)$ is not in $\{v',w'\}$. Without loss of generality, say $\tau^{-1}(y_1)\not\in\{v',w'\}$. Then $y=y_1$ is the desired vertex. 

The proof of Claim 2 is similar. There are strictly fewer than $r-(2r-2\delta)=2\delta-r$ vertices in $B_Y\cap\tau(B_X)$, at most $r-\delta$ vertices in $B_X$ that are not adjacent to $\tau^{-1}(q)$, and at most $r-\delta$ vertices in $A_Y$ that are not adjacent to $q$. Since $(2\delta-r)+(r-\delta)+(r-\delta)=r$, the desired vertex $s$ must exist. 

Now suppose $\tau$ is a bijection that can be obtained from $\sigma$ via a sequence of $(X,Y)$-friendly swaps that does not involve $u$, $v$, or $w$. This implies that $\tau(u')=\sigma(u')=u$, $\tau(v')=\sigma(v')=v$, and $\tau(w')=\sigma(w')=w$. Let us assume for the moment that at least one of the following is true: 
\begin{enumerate}[I.]
\item There exists $x\in\tau(D)\cap A_Y$.
\item We have $|\tau^{-1}(B_Y)\cap A_X|>2r-2\delta$.
\end{enumerate} 

Suppose we are in Case~I. If $|\tau^{-1}(A_Y)\cap A_X|>2r-2\delta+1$, then by Claim~1, there exists $y\in (B_Y\cap\tau(B_X))\setminus\{v,w\}$ such that we can apply the $(X,Y)$-friendly swap $xy$ to $\tau$. We call this a \emph{preliminary swap of the first kind}. This swap does not involve $u$, $v$, or $w$. To see this, note that $x$ is in $\tau(D)$, which is disjoint from $\{u,v,w\}$. Also, we chose $y$ so that it is not in $\{v,w\}$, and we must have $y\neq u$ since $y\in B_Y$ and $u\in A_Y$. If instead we have $|\tau^{-1}(A_Y)\cap A_X|\leq 2r-2\delta+1$, then $|\tau^{-1}(B_Y)\cap A_X|\geq r-(2r-2\delta+1)=2\delta-r-1>2r-2\delta$, which means that we are in Case~II. 

Suppose now that we are in Case~II. Since $|D|\leq 2r-2\delta$, there is at least one vertex $q\in B_Y\cap \tau(A_X \setminus D)$. By Claim 2, there is a vertex $s\in A_Y\cap\tau(B_X)$ such that we can apply the $(X,Y)$-friendly swap $qs$ to $\tau$. We call this a \emph{preliminary swap of the second kind}. This swap does not involve $u$, $v$, or $w$. Indeed, $q$ is in $B_Y\cap\tau(A_X)$, which is disjoint from $\{u,v,w\}$, and $s$ is in $A_Y\cap\tau(B_X)$, which is also disjoint from $\{u,v,w\}$. 

If $\tau'$ is obtained by applying a preliminary swap of the first kind to $\tau$, then $|\tau'(D)\cap A_Y|<|\tau(D)\cap A_Y|$ and $|B_Y\cap\tau'(A_X\setminus D)|\leq |B_Y\cap\tau(A_X\setminus D)|$. If $\tau'$ is obtained by applying a preliminary swap of the second kind to $\tau$, then $|B_Y\cap\tau'(A_X\setminus D)|<|B_Y\cap\tau(A_X\setminus D)|$ and $|\tau'(D)\cap A_Y|\leq|\tau(D)\cap A_Y|$.  Therefore, if we repeatedly apply these preliminary swaps, we eventually obtain a bijection $\mu: V(X) \to V(Y)$ such that 
\begin{itemize}
\item $\mu$ can be obtained from $\sigma$ by performing a sequence of swaps not involving $u$, $v$, or $w$;
\item $\mu(u')=u$, $\mu(v')=v$, and $\mu(w')=w$; 
\item $\mu(D)\subseteq B_Y$ (i.e., $\mu(D)\cap A_Y=\emptyset$); 
\item $B_Y\cap\mu(A_X\setminus D)=\emptyset$.
\end{itemize}
Notice that the condition $B_Y\cap\mu(A_X\setminus D)=\emptyset$ forces $|\mu^{-1}(B_Y) \cap A_X|\leq |D|\leq 2r-2\delta$. 

Let $\Sigma$ be a sequence of $(X,Y)$-friendly swaps not involving $u$, $v$, or $w$ that transforms $\sigma$ into $\mu$. We will demonstrate that there is a sequence $\widetilde \Sigma$ of $(X,Y)$-friendly swaps such that applying $\widetilde\Sigma$ to $\mu$ exchanges $u$ and $v$. It will then follow that $\Sigma^*=\Sigma,\widetilde\Sigma,\rev(\Sigma)$ is a sequence of $(X,Y)$-friendly swaps such that applying $\Sigma^*$ to $\sigma$ exchanges $u$ and $v$ (i.e., $\Sigma^*$ transforms $\sigma$ into $\sigma\circ(u'\,\,v')$); this will complete the proof. 

\begin{figure}[ht]
\begin{center}
\includegraphics[height=2.7cm]{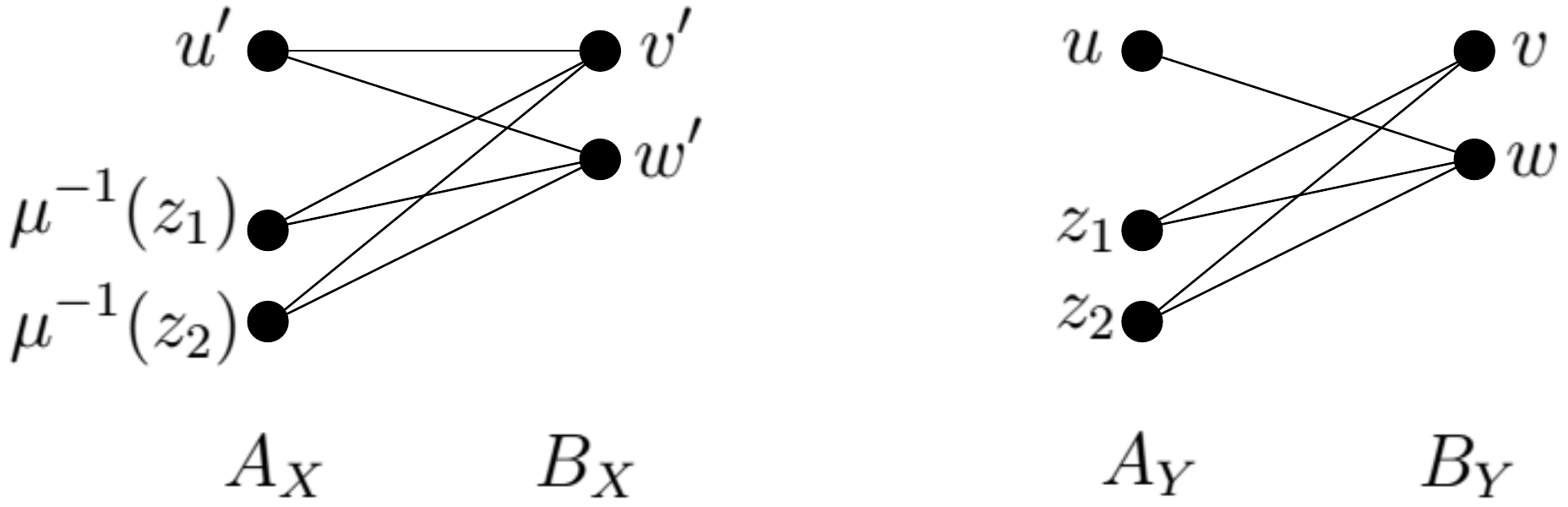}
\caption{The special vertices and edges used in the proof of Proposition~\ref{propmindegreebipartite}.}
\label{Fig2}
\end{center}  
\end{figure}

We wish to choose distinct $z_1,z_2\in N(v)\cap N(w)$ such that $\mu^{-1}(z_1)$ and $\mu^{-1}(z_2)$ are both in $A_X$. To see that this is possible, note that there are at most $2r-2\delta$ vertices in $\mu^{-1}(B_Y)\cap A_X$, at most $r-\delta$ vertices in $A_Y$ that are not adjacent to $v$, and at most $r-\delta$ vertices in $A_Y$ that are not adjacent to $w$. Because $(2r-2\delta)+(r-\delta)+(r-\delta)=4r-4\delta < r-1$, the desired vertices $z_1,z_2$ exist. Now recall that $\mu(D)\subseteq B_Y$ so that $\mu^{-1}(z_1)$ and $\mu^{-1}(z_2)$ are both in $A_X \setminus D$. By the definition of $D$, the vertices $\mu^{-1}(z_1)$ and $\mu^{-1}(z_2)$ are both adjacent to $v'$ and $w'$. See Figure~\ref{Fig2} for an illustration of the graphs $X$ and $Y$ along with the special vertices that we have considered up to this point. We now readily check that applying the sequence 
\[\widetilde\Sigma=vz_1, wz_2, wz_1, wu, wz_2, wz_1, vz_1, vz_2, wz_2\] to $\mu$ exchanges $u$ and $v$.
\end{proof}

\section{Future Work}\label{Sec:Conclusion}

Of course, it would be desirable to improve the estimates obtained in our main theorems. Along these lines, we have the following conjectures. In Conjectures~\ref{Conj3} and \ref{Conj4}, we preserve the definitions of $d_n$ and $d_{r,r}$ from Theorems~\ref{Thm:mindegree} and \ref{Thm:bipartitemindegree}, respectively.

\begin{conjecture}
There exists an absolute constant $C>0$ such that if $p\geq Cn^{-1/2}$ and $X$ and $Y$ are independently-chosen random graphs in $\mathcal G(n,p)$, then $\FS(X,Y)$ is connected with high probability.
\end{conjecture}

\begin{conjecture}
There exists an absolute constant $C>0$ such that if $p\geq Cr^{-1/2}$ and $X$ and $Y$ are independently-chosen random graphs in $\mathcal G(K_{r,r},p)$, then $\FS(X,Y)$ has exactly $2$ connected components with high probability.
\end{conjecture}

\begin{conjecture}\label{Conj3}
We have $\displaystyle d_n=\dfrac{3}{5}n+O(1)$.
\end{conjecture}

\begin{conjecture}\label{Conj4}
We have $\displaystyle d_{r,r}=\left\lceil\frac{3r+1}{4}\right\rceil$.
\end{conjecture}

We mention that it is also possible to study non-symmetric versions of the main questions that we have investigated in this paper.  One possible extension of the ``typical'' problem is to gain information about the pairs of probabilities $(p_1(n),p_2(n))$ such that $\FS(X,Y)$ is connected with high probability when $X$ and $Y$ are drawn from $\mathcal{G}(n,p_1)$ and $\mathcal{G}(n,p_2)$, respectively.  Similarly, for the ``extremal'' problem, one could ask about the pairs $(\delta_1(n), \delta_2(n))$ that guarantee the connectedness of $\FS(X,Y)$ whenever $X$ and $Y$ are $n$-vertex graphs with minimum degrees at least $\delta_1$ and $\delta_2$, respectively.  The bipartite analogues of both of these questions could also be interesting.

It may also be fruitful to study a more fine-grained ``hitting-time'' version of the result in Theorem~\ref{thm:random}, which says that the threshold for $\FS(X,Y)$ becoming connected with high probability is roughly the same as the threshold for $\FS(X,Y)$ having no isolated vertices with high probability.  To be precise, fix a positive integer $n$, and let $\{(X_t,Y_t)\}_{0 \leq t \leq \binom{n}{2}}$ be a random sequence of pairs of $n$-vertex graphs, where $X_0$ and $Y_0$ have no edges and each $X_t$ (respectively, $Y_t$) is obtained from $X_{t-1}$ (respectively, $Y_{t-1}$) by independently at random adding an edge that is not already in $X_t$ (respectively, $Y_t$).  Note that $X_{\binom{n}{2}}=Y_{\binom{n}{2}}=K_n$, so $\FS\left(X_{\binom{n}{2}},Y_{\binom{n}{2}}\right)$ is certainly connected.  Let $t_{\text{iso}}$ be the smallest value of $t$ for which $\FS(X_t,Y_t)$ has no isolated vertices, and let $t_{\text{conn}}$ be the smallest value of $t$ for which $\FS(X_t,Y_t)$ is connected.  It is obvious that $t_{\text{conn}} \geq {t_\text{iso}}$; Theorem~\ref{thm:random} and Proposition~\ref{Prop:NoIsolated} show that $t_{\text{conn}} \leq {t_\text{iso}} \cdot n^{o(1)}$ with high probability.  We might ask if $t_{\text{iso}}$ and $t_{\text{conn}}$ are more closely related.

\begin{question}\label{ques:hitting-time}
Is it true that $t_{\operatorname{conn}}=t_{\operatorname{iso}}$ with high probability?  If not, is it at least true that $t_{\operatorname{conn}}=t_{\operatorname{iso}} (1+o(1))$ with high probability, or that $t_{\operatorname{conn}}=O(t_{\operatorname{iso}})$ with high probability?
\end{question}

In a similar direction, suppose again that $X$ and $Y$ are chosen independently from $\mathcal{G}(n,p)$. One might investigate how the expected number and sizes of the connected components of $\FS(X,Y)$ change as $p$ varies.   For instance, is there a phenomenon akin to the giant component phenomenon for random graphs?

We know from Proposition~\ref{prop:complete-bipartite} that $\FS(K_{r,r},K_{r,r})$ has exactly $2$ connected components; it would be interesting to understand which edge-subgraphs $X$ of $K_{r,r}$ are such that $\FS(X,K_{r,r})$ has exactly $2$ connected components.  (Recall that in the non-bipartite case, $\FS(X, K_n)$ is connected if and only if $X$ is connected.) 

\begin{question}
Let $X$ be a random graph in $\mathcal G(K_{r,r,},p)$. Under what conditions on $p$ will $\FS(X,K_{r,r})$ be disconnected with high probability? Under what conditions on $p$ will $\FS(X,K_{r,r})$ be connected with high probability? 
\end{question}

\begin{problem}
Define $d_{r,r}^*$ to be the smallest nonnegative integer such that for every edge-subgraph $X$ of $K_{r,r}$ with $\delta(X)\geq d_{r,r}^*$, the graph $\FS(X,K_{r,r})$ has exactly $2$ connected components. Obtain estimates for $d_{r,r}^*$. 
\end{problem}

We have focused on the number of connected components of the friends-and-strangers graphs $\FS(X,Y)$, but one could also consider other graph parameters. Most notably, it would be very interesting to have nontrivial results concerning the diameters of these graphs. 
\begin{question}\label{Conj5}
Does there exist an absolute constant $C>0$ such that for all $n$-vertex graphs $X$ and $Y$, every connected component of $\FS(X,Y)$ has diameter at most $n^{C}$? 
\end{question}

\begin{problem}
Obtain estimates (in terms of $n$ and $p$) for the expected value of the maximum diameter of a connected component of $\FS(X,Y)$ when $X$ and $Y$ are independently-chosen random graphs in $\mathcal G(n,p)$. 
\end{problem}

Finally, we mention that it could be fruitful to study random walks on friends-and-strangers graphs; indeed, this corresponds to friends and strangers randomly walking on graphs. Random walks on $\FS(X,K_n)$ correspond to the interchange process on $X$ as discussed, for example, in \cite{aldous}. 

\section*{Acknowledgments}
The first author is supported in part by NSF grant DMS--1855464, BSF grant 2018267, and the Simons Foundation.  The second author is supported by an NSF Graduate Research Fellowship (grant DGE--1656466) and a Fannie and John Hertz Foundation Fellowship.  We are grateful to Kiril Bangachev for pointing out an error in an earlier proof of Lemma~\ref{Lem:9/14}.

\end{document}